\documentclass[a4paper,11pt]{article}
\usepackage{amsmath,amsthm}
\usepackage{authblk}
\usepackage[style=numeric-comp,maxbibnames=99,bibencoding=utf8,firstinits=true]{biblatex}
\usepackage{booktabs}
\usepackage{cancel}
\usepackage{enumitem}
\usepackage[hmargin={26mm,26mm},vmargin={30mm,35mm}]{geometry}
\usepackage{nicefrac}
\usepackage[colorlinks, allcolors=blue]{hyperref}
\usepackage{loglogslopetriangle}
\usepackage{subcaption}
\usepackage[compact,small]{titlesec}
\usepackage{tikz}
\usepackage{tikz-cd}

\usepackage{newtxtext}
\usepackage{newtxmath}

\bibliography{divdivsymm}
\AtBeginBibliography{\footnotesize}

\newcommand{\email}[1]{\href{mailto:#1}{#1}}

\allowdisplaybreaks

\newtheorem{theorem}{Theorem}
\newtheorem{proposition}[theorem]{Proposition}
\newtheorem{lemma}[theorem]{Lemma}

\theoremstyle{remark}
\newtheorem{remark}[theorem]{Remark}
\theoremstyle{definition}

\newcommand{\st}{\,:\,}
\newcommand{\Real}{\mathbb{R}}
\newcommand{\Natural}{\mathbb{N}}
\newcommand{\Symm}{\mathbb{S}}

\DeclareRobustCommand{\bvec}[1]{\boldsymbol{#1}}
\newcommand{\uvec}[1]{\underline{\bvec{#1}}}
\newcommand{\cvec}[1]{\bvec{\mathcal{#1}}}

\DeclareRobustCommand{\btens}[1]{\boldsymbol{#1}}
\newcommand{\utens}[1]{\underline{\bvec{#1}}}
\newcommand{\ctens}[1]{\bvec{\mathcal{#1}}}

\DeclareMathOperator{\card}{card}
\DeclareMathOperator{\SPAN}{span}
\DeclareMathOperator{\DIM}{dim}

\DeclareMathOperator{\tr}{tr}

\DeclareMathOperator{\GRAD}{\bf grad}
\DeclareMathOperator{\CURL}{\bf curl}
\DeclareMathOperator{\DIV}{div}
\DeclareMathOperator{\VDIV}{\bf div}

\DeclareMathOperator{\VROT}{\bf rot}
\DeclareMathOperator{\SYM}{sym}
\DeclareMathOperator{\HESS}{\bf hess}

\newcommand{\sym}{{\rm sym}}

\newcommand{\RT}[1]{\boldsymbol{\mathcal{RT}}^{#1}}

\newcommand{\Hdivdiv}[2]{\bvec{H}(\DIV\VDIV,#1;#2)}

\newcommand{\compl}{{\rm c}}

\newcommand{\Poly}[2][]{\mathcal{P}_{#1}^{#2}}
\newcommand{\vPoly}[2][]{\cvec{P}_{#1}^{#2}}
\newcommand{\tPoly}[2][]{\ctens{P}_{#1}^{#2}}

\newcommand{\Holy}[1]{\cvec{H}^{#1}}
\newcommand{\cHoly}[1]{\cvec{H}^{\compl,#1}}

\newcommand{\edges}[1]{\mathcal{E}_{#1}}
\newcommand{\vertices}[1]{\mathcal{V}_{#1}}

\newcommand{\ET}{\edges{T}}
\newcommand{\VE}{\vertices{E}}
\newcommand{\VT}{\vertices{T}}

\newcommand{\normal}{\bvec{n}}
\newcommand{\tangent}{\bvec{t}}

\newcommand{\Mh}{\mathcal{M}_h}
\newcommand{\Th}{\mathcal{T}_h}
\newcommand{\Eh}{\mathcal{E}_h}
\newcommand{\Vh}{\mathcal{V}_h}

\DeclareMathOperator{\Ker}{Ker}
\DeclareMathOperator{\Image}{Im}

\newcommand{\norm}[2][]{\|#2\|_{#1}}
\newcommand{\seminorm}[2][]{|#2|_{#1}}
\newcommand{\vvvert}{\vert\kern-0.25ex\vert\kern-0.25ex\vert}
\newcommand{\tnorm}[2][]{\vvvert #2\vvvert_{#1}}

\newcommand{\term}{\mathfrak{T}}

\newcommand{\IV}[2][T]{\uvec{I}_{\bvec{V},#1}^{#2}}
\newcommand{\ISigma}[2][T]{\utens{I}_{\btens{\Sigma},#1}^{#2}}

\newcommand{\Csym}[1]{\btens{\mathsf{C}}_{\sym,T}^{#1}}
\newcommand{\DD}[1]{\mathsf{DD}_T^{#1}}
\newcommand{\uCsym}[2][T]{\utens{C}_{\sym,#1}^{#2}}
\newcommand{\PSigmaT}[1]{\btens{P}_{\btens{\Sigma},T}^{#1}}
\newcommand{\PSigmaE}[1]{P_{\btens{\Sigma},E}^{#1}}

\begin{document}

\title{A fully discrete plates complex on polygonal meshes with application to the Kirchhoff--Love problem}
\author[1]{Daniele A. Di Pietro}
\author[2]{J\'{e}r\^{o}me Droniou}
\affil[1]{IMAG, Univ Montpellier, CNRS, Montpellier, France, \email{daniele.di-pietro@umontpellier.fr}}
\affil[2]{School of Mathematics, Monash University, Melbourne, Australia, \email{jerome.droniou@monash.edu}}

\maketitle

\begin{abstract}
  In this work we develop a novel fully discrete version of the plates complex, an exact Hilbert complex relevant for the mixed formulation of fourth-order problems.
  The derivation of the discrete complex follows the discrete de Rham paradigm, leading to an arbitrary-order construction that applies to meshes composed of general polygonal elements.
  The discrete plates complex is then used to derive a novel numerical scheme for Kirchhoff--Love plates, for which a full stability and convergence analysis are performed.
  Extensive numerical tests complete the exposition.
\end{abstract}

\section{Introduction}

Denote by $T\subset\Real^2$ a contractible polygonal set and by $\Symm$ the space of symmetric $2\times 2$ matrices.
In this paper we develop a fully discrete counterpart of the following exact Hilbert complex \cite{Chen.Hu.ea:18}:
\begin{equation}\label{eq:continuous.complex}
  \begin{tikzcd}
    \RT{1}(T)
    \arrow[r,hook] & \bvec{H}^1(T;\Real^2)
    \arrow{r}[above=2pt]{\SYM\CURL} & \Hdivdiv{T}{\Symm}
    \arrow{r}[above=2pt]{\DIV\VDIV} & L^2(T)
    \arrow{r}[above=2pt]{0} & 0,
  \end{tikzcd}
\end{equation}
where ``$\SYM$'' denotes the symmetric part of a space or an operator, while $\RT{1}(T) \coloneq \vPoly{0}(T) + \bvec{x}\Poly{0}(T)$ is the lowest-order Raviart--Thomas space \cite{Raviart.Thomas:77}.
For a precise definition of the differential operators $\SYM\CURL$ and $\DIV\VDIV$ in Cartesian coordinates, we refer to \eqref{eq:2d.differential.operators} below.
As the complex \eqref{eq:continuous.complex} is relevant for mixed formulations of Kirchhoff--Love plates \cite{Chen.Hu.ea:18}, it will be referred to as \emph{plates complex} in what follows.
\smallskip

The fully discrete counterpart of the plates complex proposed in this work follows the Discrete de Rham (DDR) paradigm \cite{Di-Pietro.Droniou.ea:20,Di-Pietro.Droniou:21*1} (see also \cite{Di-Pietro.Droniou:21}), and it appears to be the first one designed to support arbitrary order and general meshes possibly including polygonal elements and non-matching interfaces.
The principle of the DDR paradigm is to replace both spaces and operators by discrete counterparts designed so as to be compatible with the cohomology properties of the continuous complex.
Specifically:
\begin{itemize}[left=0pt,topsep=2pt,parsep=0pt,itemsep=2pt]
\item The \emph{discrete spaces} are spanned by vectors of polynomials with components attached to mesh entities in order to mimic, through their single-valuedness, global continuity properties of the continuous spaces.
  The polynomial space for each component is selected to ensure compatibility with the continuous complex, and can be (a) a full polynomial space, (b) a space obtained applying a differential operator to a full polynomial space, or (c) its Koszul complement (see \cite[Chapter 7]{Arnold:18} on this concept).
  Cases (b) and (c) correspond to \emph{incomplete} polynomial spaces.
\item The \emph{discrete operators} are obtained in two steps: first, operator reconstructions in full polynomial spaces are built mimicking an appropriate integration by parts formula; second, whenever needed, the $L^2$-orthogonal projection on the appropriate incomplete polynomial space is taken.
\end{itemize}
In this work, these general principles are declined for the plates complex \eqref{eq:continuous.complex}.
A key additional difficulty with respect to the de Rham complex considered in \cite{Di-Pietro.Droniou.ea:20,Di-Pietro.Droniou:21*1} is that both the spaces and operators have to account for the additional algebraic constraints resulting from symmetry.
\smallskip

The novel polygonal plates complex developed in the first part of this work is applied to Kirchhoff--Love plates (see \cite{Di-Pietro.Droniou:21*2} concerning the application of the DDR paradigm to Reissner--Mindlin plates).
The support of polygonal meshes in the context of solid mechanics has several interests, including the possibility to perform adaptation through nonconforming mesh refinement or agglomeration \cite{Bassi.Botti.ea:12,Antonietti.Giani.ea:13}, as well as the support of easy mesh cutting, e.g., for the modelling of cracks.
Denote by $\Omega$ a contractible polygonal domain of $\Real^2$ corresponding to the surface of the plate in its reference configuration.
Given an orthogonal load $f:\Omega\to\Real$, we seek the \emph{moment tensor} $\btens{\sigma}:\Omega\to\Symm$ and the \emph{deflection} $u:\Omega\to\Real$ such that
\begin{subequations}\label{eq:strong.problem}
  \begin{alignat}{4}\label{eq:strong.problem:constitutive.law}
    \btens{\sigma} + \mathbb{A}\HESS u &= 0 &\qquad&\text{in $\Omega$},
    \\ \label{eq:strong.problem:equilibrium}
    -\DIV\VDIV\btens{\sigma} &= f  &\qquad&\text{in $\Omega$},
    \\ \label{eq:strong.problem:bc}
    u = \partial_{\normal} u &= 0 &\qquad&\text{on $\partial\Omega$},
  \end{alignat}
\end{subequations}
with $\HESS$ denoting the Hessian operator and $\mathbb{A}$ the fourth-order tensor defined by $\mathbb{A}\btens{\tau} = D\big[(1-\nu)\btens{\tau} + \nu\tr(\btens{\tau})\btens{I}_2\big]$ for all $\btens{\tau}\in\Symm$, where $D$ is the bending modulus, $\nu$ the Poisson ratio, and $\btens{I}_2$ is the $2\times 2$ identity matrix.
Notice that the symmetry requirement on $\btens{\sigma}$ naturally results from \eqref{eq:strong.problem:constitutive.law} owing to the Schwarz theorem.
A weak formulation of \eqref{eq:strong.problem} with symmetric moment tensors in a space embedding the continuity of the normal-normal component at interfaces underpins the Hellan--Herrmann--Johnson method, the analysis of which has been considered in several works \cite{Brezzi.Marini:75,Brezzi.Raviart:77,Brezzi.Marini.ea:80,Babuska.Osborn.ea:80,Arnold.Brezzi:85,Comodi:89,Blum.Rannacher:90,Stenberg:91,Krendl.Rafetseder.ea:16,Rafetseder.Zulehner:18}; see also \cite{Arnold.Walker:20} concerning the application to domains with curved boundaries and \cite{Neunteufel.Schoberl:19} for nonlinear shells.
Recently, a weak formulation based on moment tensors in $\Hdivdiv{\Omega}{\Symm}$, along with the corresponding finite element discretisation, has been studied in \cite{Chen.Huang:20}.
While the aforementioned work has inspired the developments of the present paper, the discrete finite element complexes constructed therein are restricted to matching triangular meshes (and, even on these meshes, differ from the one proposed here).
For the sake of completeness, we also mention that polygonal methods for the primal formulation of Kirchhoff--Love plates have been developed in \cite{Bonaldi.Di-Pietro.ea:18,Antonietti.Manzini.ea:18}.
\smallskip

The rest of this work is organised as follows.
In Section \ref{sec:setting} we establish the setting (differential operators, notation for geometric entities, polynomial spaces).
Section \ref{sec:complex} is devoted to the construction of the local discrete complex and the proof of its exactness for contractible polygonal elements.
The application to Kirchhoff--Love plates is considered in Section \ref{sec:application}, where complete stability and convergence analysis are carried out, and numerical examples on various meshes are presented.


\section{Setting}\label{sec:setting}

\subsection{Two-dimensional vector calculus operators}

Consider the real plane $\Real^2$ endowed with the Cartesian coordinate system $(x_1,x_2)$.
We will need the following two-dimensional differential operators acting on smooth enough
scalar-valued fields $q$,
vector-valued fields $\bvec{v}=\begin{pmatrix}v_1\\v_2\end{pmatrix}$,
or matrix-valued fields $\btens{\tau}=\begin{pmatrix}\tau_{11} & \tau_{12}\\ \tau_{21} & \tau_{22}\end{pmatrix}$:
\begin{equation}\label{eq:2d.differential.operators}
\begin{gathered}
  \CURL q\coloneq\begin{pmatrix}\partial_2 q\\ -\partial_1 q\end{pmatrix},
  \\
  \DIV\bvec{v}\coloneq\partial_1 v_1 + \partial_2 v_2,\quad
  \GRAD\bvec{v}\coloneq\begin{pmatrix}
  \partial_1 v_1 & \partial_2 v_1 \\
  \partial_1 v_2 & \partial_2 v_2
  \end{pmatrix},\quad
  \SYM\CURL\bvec{v}\coloneq\begin{pmatrix}
  \partial_2 v_1 & \frac{-\partial_1 v_1+\partial_2 v_2}{2} \\
  \frac{-\partial_1 v_1+\partial_2 v_2}{2} & -\partial_1 v_2
  \end{pmatrix},
  \\
  \VDIV\btens{\tau}\coloneq\begin{pmatrix}
  \partial_1\tau_{11} + \partial_2\tau_{12}
  \\
  \partial_1\tau_{21} + \partial_2\tau_{22}
  \end{pmatrix},\quad
  \VROT\btens{\tau}\coloneq\begin{pmatrix}
  \partial_2\tau_{11} - \partial_1\tau_{12} \\
  \partial_2\tau_{21} - \partial_1\tau_{22}
  \end{pmatrix},
\end{gathered}
\end{equation}
where $\partial_i$ denotes the partial derivative with respect to the $i$th coordinate.
Defining the fourth-order tensor $\mathbb{C}$ such that
\begin{equation}\label{eq:def.tensor.C}
  \mathbb{C}\btens{\tau}=\begin{pmatrix}\tau_{12} & \frac{-\tau_{11}+\tau_{22}}{2}\\ \frac{-\tau_{11}+\tau_{22}}{2} & -\tau_{21}\end{pmatrix}\qquad\forall \btens{\tau}=\begin{pmatrix}\tau_{11} & \tau_{12}\\ \tau_{21} & \tau_{22}\end{pmatrix}\in\Real^{2\times 2},
\end{equation}
we have $\SYM\CURL\bvec{v} = \mathbb{C}\GRAD\bvec{v}$.

\subsection{Notation for geometric entities}

Let $T\subset\Real^2$ be a contractible polygonal domain and denote by $\bvec{x}_T$ a point inside $T$ such that there exists a disk centered in $\bvec{x}_T$ and contained in $T$.
The sets of edges and vertices of $T$ are denoted by $\ET$ and $\VT$, respectively.
For each edge $E\in\ET$, we denote by $\VE$ the set of vertices corresponding to its endpoints and fix an orientation by prescribing a unit tangent vector $\tangent_E$.
This orientation determines two numbers $(\omega_{EV})_{V\in\VE}$ in $\{-1,+1\}$ such that $\omega_{EV}=+1$ whenever $\tangent_E$ points towards $V$.
The corresponding unit normal vector $\normal_E$ is selected so that $(\tangent_E,\normal_E)$ forms a right-handed system of coordinates, and we denote by $\omega_{TE}\in\{-1,+1\}$ the orientation of $E$ relative to $T$, defined so that $\omega_{TE}\normal_E$ points out of $T$.

\subsection{Polynomial spaces}

Given $Y\in\{T\}\cup\ET$, we denote by $\Poly{\ell}(Y)$ the space spanned by the restriction to $Y$ of two-variate polynomials of total degree $\le\ell$,  with the convention that $\Poly{-1}(Y)=\{0\}$.
The corresponding $L^2$-orthogonal projector is denoted by $\pi_{\mathcal{P},Y}^\ell$.
The symbols $\vPoly{\ell}(Y;\Real^2)$ and $\tPoly{\ell}(Y;\Symm)$ denote, respectively, vector-valued and symmetric tensor-valued functions over $T$ whose components are in $\Poly{\ell}(Y)$.
Notice that, for all $E\in\ET$, the space $\Poly{\ell}(E)$ is isomorphic to univariate polynomials of total degree $\le\ell$ (see \cite[Proposition 1.23]{Di-Pietro.Droniou:20}).
In what follows, with a little abuse of notation, both spaces are denoted by $\Poly{\ell}(E)$.
Finally, we denote by $\Poly{\ell}(\ET)$ the space of broken polynomials of total degree $\le\ell$ on $\ET$.
Vector and tensor versions of this space are denoted in boldface and the codomain is specified.


\section{A local fully discrete complex}\label{sec:complex}

\subsection{Spaces}

We will need the following decomposition of the space $\tPoly{\ell}(T;\Symm)$ of symmetric tensor-valued polynomials of total degree $\le m$:
\begin{equation}\label{eq:PolySym=Holy.oplus.cHoly}
  \begin{gathered}
    \tPoly{m}(T;\Symm)
    = \Holy{m}(T)
    \oplus\cHoly{m}(T)
    \\
    \text{
      with $\Holy{m}(T)\coloneq\HESS\Poly{m+2}(T)$ and
      $\cHoly{m}(T)\coloneq\SYM\big(
      (\bvec{x} - \bvec{x}_T)^\bot\otimes\vPoly{m-1}(T;\Real^2)
      \big)$,
    }
  \end{gathered}
\end{equation}
where $\SYM\btens{\tau}=\frac{\btens{\tau}+\btens{\tau}^\top}{2}$ is the symmetrisation operator.
The $L^2$-orthogonal projectors on $\Holy{m}(T)$ and $\cHoly{m}(T)$ are respectively denoted by $\btens{\pi}_{\cvec{H},T}^m$ and $\btens{\pi}_{\cvec{H},T}^{\compl,m}$, and we notice that, since the kernel of the Hessian operator coincides with the space of affine bivariate polynomials (that has dimension 3) and $\DIM\Poly{r}(T) = \tfrac12(r+2)(r+1)$,
\begin{equation}\label{eq:dim.Holy.cHoly}
  \DIM(\Holy{m}(T)) = \tfrac12(m+4)(m+3) - 3,\qquad
  \DIM(\cHoly{m}(T)) = m(m+1).
\end{equation}
We let, for any $k\ge 3$ and $\ell\ge 2$,
\begin{subequations}\label{eq:local.spaces}
  \begin{align}\label{eq:VT}
    \uvec{V}_T^k&\coloneq\Big\{
    \begin{aligned}[t]
      &\uvec{v}_T=
      \big(
      \bvec{v}_T,
      (\bvec{v}_E)_{E\in\ET},
      (\bvec{v}_V, \btens{G}_{\bvec{v},V})_{V\in\VT}
      \big)\st
      \\
      &\qquad\text{
        $\bvec{v}_T\in\vPoly{k-2}(T;\Real^2)$,
      }
      \\
      &\qquad\text{
        $\bvec{v}_E\in\vPoly{k-4}(E;\Real^2)$ for all $E\in\ET$,
      }
      \\
      &\qquad\text{      
        $\bvec{v}_V\in\Real^2$
        and $\btens{G}_{\bvec{v},V}\in\Real^{2\times 2}$
        for all $V\in\VT$
      }
      \Big\},
    \end{aligned}
    \\\label{eq:SigmaT}
    \utens{\Sigma}_T^\ell&\coloneq\Big\{
    \begin{aligned}[t]
      &\utens{\tau}_T
      =\big(
      \btens{\tau}_{\cvec{H},T}, \btens{\tau}_{\cvec{H},T}^\compl,
      (\tau_E,D_{\bvec{\tau},E})_{E\in\ET}
      (\btens{\tau}_V)_{V\in\VT}
      \big)\st
      \\
      &\qquad\text{
        $(\btens{\tau}_{\cvec{H},T},\btens{\tau}_{\cvec{H},T}^\compl)\in\Holy{\ell-3}(T)\times\cHoly{\ell}(T)$,
      }
      \\
      &\qquad\text{
        $\tau_E\in\Poly{\ell-2}(E)$
        and $D_{\bvec{\tau},E}\in\Poly{\ell-1}(E)$ for all $E\in\ET$,
      }
      \\
      &\qquad\text{
        $\btens{\tau}_V\in\Symm$ for all $V\in\VT$
      }
      \Big\}.
    \end{aligned}
  \end{align}
\end{subequations}
\begin{remark}[Dimensions of the local spaces]
  Recalling that, for all $E\in\ET$, $\DIM\Poly{m}(E) = m+1$, it holds
  \begin{equation}\label{eq:dim.VT}
    \begin{aligned}
      \DIM\uvec{V}_T^k
      &= k(k-1) + 2(k-3)\card(\ET) + 6\card(\VT)
      \\
      &=  k(k-1) + 2k\card(\VT),
    \end{aligned}
  \end{equation}
  where the conclusion follows observing that $\card(\VT) = \card(\ET)$.
  Further recalling \eqref{eq:dim.Holy.cHoly}, we have
  \begin{equation}\label{eq:dim.SigmaT}
    \begin{aligned}
      \DIM\utens{\Sigma}_T^\ell
      &= \tfrac12\ell(\ell+1)-3 + \ell(\ell+1) + \left[(\ell-1) + \ell\right]\card(\ET) + 3\card(\VT)
      \\
      &= \tfrac32\ell(\ell+1) + 2(\ell + 1)\card(\VT) - 3.
    \end{aligned}
  \end{equation}
\end{remark}
The interpolators on $\uvec{V}_T^k$ and $\utens{\Sigma}_T^\ell$ are, respectively,
$\IV{k}:\bvec{C}^1(\overline{T};\Real^2)\to\uvec{V}_T^k$ and
$\ISigma{\ell}:\btens{H}^2(T;\Symm)\to\utens{\Sigma}_T^\ell$
such that, for all $\bvec{v}\in\bvec{C}^1(\overline{T};\Real^2)$ and all $\btens{\tau}\in\btens{H}^2(T;\Symm)$,
\begin{align}\label{eq:IVT}
  \IV{k}\bvec{v}&\coloneq
  \Big(
  \bvec{\pi}_{\cvec{P},T}^{k-2}\bvec{v},
  (\bvec{\pi}_{\cvec{P},E}^{k-4}\bvec{v}_{|E})_{E\in\ET},
  \big(\bvec{v}(\bvec{x}_V),\GRAD\bvec{v}(\bvec{x}_V)\big)_{V\in\VT}
  \Big),
  \\ \label{eq:ISigmaT}
  \ISigma{\ell}\btens{\tau}&\coloneq
  \Big(
  \begin{aligned}[t]
  &\btens{\pi}_{\cvec{H},T}^{\ell-3}\btens{\tau},
    \btens{\pi}_{\cvec{H},T}^{\compl,\ell}\btens{\tau},
    \\
  &\big(
  \pi_{\cvec{P},E}^{\ell-2}(\btens{\tau}_{|E}\normal_E\cdot\normal_E),
  \pi_{\cvec{P},E}^{\ell-1}\big[
  \partial_{\tangent_E}(\btens{\tau}_{|E}\normal_E\cdot\tangent_E)
  + (\VDIV\btens{\tau})_{|E}\cdot\normal_E
  \big]
  \big)_{E\in\ET},
  \\
  &\big(\btens{\tau}(\bvec{x}_V)\big)_{V\in\VT}
  \Big),
  \end{aligned}
\end{align}
where $\bvec{x}_V$ denotes the coordinate vector of the vertex $V\in\VT$
while, for all $E\in\ET$, $\partial_{\tangent_E}$ denotes the derivative along the edge $E$ in the direction of $\tangent_E$.

\subsection{Differential operators and potential reconstructions}

In what follows we define the various reconstructions of differential operators and of the corresponding vector or symmetric tensor potentials.
Following standard conventions for DDR methods, full operators mapping on complete polynomial spaces are denoted in sans serif font for easier identification.

\subsubsection{Reconstructions in $\uvec{V}_T^k$}

The key integration by parts formula to reconstruct discrete counterparts of the symmetric curl and of the corresponding vector potential is the following:
For any $\bvec{v}:T\to\Real^2$ and any $\btens{\tau}:T\to\Symm$ smooth enough,
\begin{equation}\label{eq:ibp.VT}
  \int_T\bvec{v}\cdot\VROT\btens{\tau}
  = -\int_T\SYM\CURL\bvec{v}:\btens{\tau}
  + \sum_{E\in\ET}\omega_{TE}\int_{E}\btens{\tau}\,\bvec{t}_E\cdot\bvec{v}.
\end{equation}
The \emph{full symmetric curl} $\Csym{k-1}:\uvec{V}_T^k\to\tPoly{k-1}(T;\Symm)$ is such that, for all $\uvec{v}_T\in\uvec{V}_T^k$,
\begin{equation}\label{eq:full.CsymT}
  \int_T\Csym{k-1}\uvec{v}_T:\btens{\tau}
  = -\int_T\bvec{v}_T\cdot\VROT\btens{\tau}
  + \sum_{E\in\ET}\omega_{TE}\int_E\bvec{v}_{\ET}\cdot(\btens{\tau}\,\bvec{t}_E)\qquad
  \forall\btens{\tau}\in\tPoly{k-1}(T;\Symm),
\end{equation}
where $\bvec{v}_{\ET}\in\vPoly{k}(\ET;\Real^2)\cap\bvec{C}^0(\partial T;\Real^2)$ denotes the unique function in this space such that
\begin{equation}\label{eq:v.ET}
  \begin{gathered}
    \text{
      for all $E\in\ET$,
      $\bvec{\pi}_{\cvec{P},E}^{k-4}(\bvec{v}_{\ET})_{|E} = \bvec{v}_E$ 
      and $\partial_{\tangent_E}(\bvec{v}_{\ET})_{|E}(\bvec{x}_V)=\btens{G}_{\bvec{v},V}\tangent_E$
      for all $V\in\VE$,
    }
    \\
    \text{
      and $\bvec{v}_{\ET}(\bvec{x}_V) = \bvec{v}_V$ for all $V\in\VT$.
    }
  \end{gathered}
\end{equation}
The following polynomial consistency property holds:
For all $\bvec{v}\in\vPoly{k}(T;\Real^2)$, letting $\bvec{v}_{\ET}$ be defined by \eqref{eq:v.ET} with $\uvec{v}_T = \IV{k}\bvec{v}$, we have $\bvec{v}_{\ET} = \bvec{v}_{|\partial T}$.
It can be easily checked that \eqref{eq:full.CsymT} defines $\Csym{k-1}\uvec{v}_T$ uniquely owing to the Riesz representation theorem in $\tPoly{k-1}(T;\Symm)$ equipped with the standard $L^2$-product.
By design, we have the following polynomial consistency property:
\begin{equation}\label{eq:CsymT:polynomial.consistency}
  \Csym{k-1}(\IV{k}\bvec{v}) = \SYM\CURL\bvec{v}\qquad
  \forall\bvec{v}\in\vPoly{k}(T;\Real^2).
\end{equation}
The \emph{discrete symmetric curl} $\uCsym{k-1}:\uvec{V}_T^k\to\utens{\Sigma}_T^{k-1}$, acting between the discrete spaces in the com\-plex, is obtained setting, for all $\uvec{v}_T\in\uvec{V}_T^k$,
\begin{equation}\label{eq:uCsymT}
  \uCsym{k-1}\uvec{v}_T\coloneq\Big(
  \begin{aligned}[t]
    &\btens{\pi}_{\ctens{H},T}^{k-4}\big(\Csym{k-1}\uvec{v}_T\big),
    \btens{\pi}_{\ctens{H},T}^{\compl,k-1}\big(\Csym{k-1}\uvec{v}_T\big),
    \big(
    \pi_{\mathcal{P},E}^{k-3}(\partial_{\tangent_E}\bvec{v}_{\ET}\cdot\normal_E),
    \partial_{\tangent_E}^2\bvec{v}_{\ET}\cdot\tangent_E
    \big)_{E\in\ET},
    \\
    &\big(
    \mathbb{C}\btens{G}_{\bvec{v},V}
    \big)_{V\in\VT}
    \Big),
  \end{aligned}
\end{equation}
with $\mathbb{C}$ as in \eqref{eq:def.tensor.C}.
The choice of the edge terms in \eqref{eq:uCsymT} is justified by the formulas in \cite[Lemma 2.2]{Chen.Huang:20} relating trace values of the symmetric curl with tangential derivatives of the function.
Finally, we define the \emph{vector potential} $\bvec{P}_{\bvec{V},T}^k:\uvec{V}_T^k\to\vPoly{k}(T;\Real^2)$ such that, for all $\uvec{v}_T\in\uvec{V}_T^k$,
\[
\int_T\bvec{P}_{\bvec{V},T}^k\uvec{v}_T\cdot\VROT\btens{\tau}
= -\int_T\Csym{k-1}\uvec{v}_T\cdot\btens{\tau}
+ \sum_{E\in\ET}\omega_{TE}\int_E\bvec{v}_{\ET}\cdot(\btens{\tau}\,\tangent_E)\qquad
\forall\btens{\tau}\in\cHoly{k+1}(T).
\]
To check that this condition defines $\bvec{P}_{\bvec{V},T}^k\uvec{v}_T$ uniquely, use again the Riesz representation theorem for $\vPoly{k}(T;\Real^2)$ equipped with the $L^2$-product along with the fact that $\VROT:\cHoly{k+1}(T)\to\vPoly{k}(T;\Real^2)$ is an isomorphism (see \cite[Lemma 3.6]{Chen.Huang:20}).
The following polynomial consistency property holds:
\[
\bvec{P}_{\bvec{V},T}^k\big(
\IV{k}\bvec{v}
\big) = \bvec{v}\qquad
\forall\bvec{v}\in\vPoly{k}(T;\Real^2).
\]

\subsubsection{Reconstructions in $\utens{\Sigma}_T^\ell$}

The starting point is, in this case, the following integration by parts formula, corresponding to \cite[Eq.~(2.4)]{Comodi:89} (see also \cite[Eq.~(2)]{Chen.Huang:20}) and valid for all tensor-valued functions $\bvec{\tau}:T\to\Symm$ and all scalar-valued functions $q:T\to\Real$ smooth enough:
\begin{equation}\label{eq:ibp.SigmaT}
  \begin{aligned}
    \int_T\DIV\VDIV\btens{\tau}~q
    &= \int_T\btens{\tau}:\HESS q
    - \sum_{E\in\ET}\omega_{TE}\sum_{V\in\VE}\omega_{EV}(\btens{\tau}\normal_E\cdot\tangent_E)(\bvec{x}_V)~q(\bvec{x}_V)
    \\
    &\quad
    -\sum_{E\in\ET}\omega_{TE}\left[
      \int_E(\btens{\tau}\normal_E\cdot\normal_E)~\partial_{\normal_E}q
      - \int_E\left(
      \partial_{\tangent_E}(\btens{\tau}\normal_E\cdot\tangent_E)
      + \VDIV\btens{\tau}\cdot\normal_E
      \right)q
      \right].
  \end{aligned}
\end{equation}
The \emph{discrete div-div operator} $\DD{\ell-1}:\uvec{\Sigma}_T^\ell\to\Poly{\ell-1}(T)$ is such that, for all $\utens{\tau}_T\in\utens{\Sigma}_T^\ell$,
\begin{multline}\label{eq:DDT}
\int_T \DD{\ell-1}\utens{\tau}_T~q
=
\int_T\btens{\tau}_{\cvec{H},T}:\HESS q
-\sum_{E\in\ET}\omega_{TE}\sum_{V\in\VE}\omega_{EV}\,(\btens{\tau}_V\normal_E\cdot\tangent_E)\,q(\bvec{x}_V)
\\
-\sum_{E\in\ET}\omega_{TE}\left(
  \int_E\tau_E\,\partial_{\normal_E}q
  -\int_E D_{\tau,E}\,q
  \right)
  \qquad\forall q\in\Poly{\ell-1}(T).
\end{multline}
Writing \eqref{eq:DDT} for $\utens{\tau}_T = \ISigma{\ell}\btens{\tau}$, removing the $L^2$-orthogonal projectors in the right hand side, and integrating by parts, it can be easily checked that the following commutation property holds:
\begin{equation}\label{eq:DDT.commutation}
  \DD{\ell-1}\big(\ISigma{\ell}\btens{\tau}\big)
  = \pi_{\mathcal{P},T}^{\ell-1}\big(\DIV\VDIV\btens{\tau}\big)\qquad
  \forall\btens{\tau}\in\btens{H}^2(T;\Symm).
\end{equation}
The \emph{tensor potential} $\PSigmaT{\ell}:\utens{\Sigma}_T^{\ell}\to\Poly{\ell}(T;\Symm)$ is such that, for all $\utens{\tau}_T\in\utens{\Sigma}_T^\ell$ and all $(q,\btens{\upsilon})\in\Poly{\ell+2}(T)\times\cHoly{\ell}(T)$,
\begin{equation}\label{eq:P.Sigma.T}
  \begin{aligned}
    \int_T\PSigmaT{\ell}\utens{\tau}_T:\left(\HESS q + \btens{\upsilon}\right)
    &=
    \int_T\DD{\ell-1}\utens{\tau}_T\,q
    + \sum_{E\in\ET}\omega_{TE}\sum_{V\in\VE}\omega_{EV}(\btens{\tau}_V\normal_E\cdot\tangent_E)\,q(\bvec{x}_V)
    \\
    &\qquad
    + \sum_{E\in\ET}\omega_{TE}\left(
    \int_E \PSigmaE{\ell}\utens{\tau}_E\,\partial_{\normal_E} q
    - \int_E D_{\btens{\tau},E}\,q
    \right)
    + \int_T\btens{\tau}_{\cvec{H},T}^\compl:\btens{\upsilon},
  \end{aligned}
\end{equation}
where, for all $E\in\ET$, denoting by $\utens{\tau}_E\coloneq\big(\tau_E, D_{\btens{\tau},E}, (\btens{\tau}_V)_{V\in\VE}\big)$ the restriction of $\utens{\tau}_T$ to $E$, $\PSigmaE{\ell}\utens{\tau}_E\in\Poly{\ell}(E)$ is the unique polynomial that satisfies
\begin{equation}\label{eq:P.Sigma.E}
\text{
  $\PSigmaE{\ell}\utens{\tau}_E(\bvec{x}_V) = \btens{\tau}_V\normal_E\cdot\normal_E$ for all $V\in\VE$
  and $\pi_{\mathcal{P},E}^{\ell-2}\big(\PSigmaE{\ell}\utens{\tau}_E\big) = \tau_E$.
}
\end{equation}
The fact that condition \eqref{eq:P.Sigma.T} defines $\PSigmaT{\ell}\utens{\tau}_T$ uniquely follows from the Riesz representation theorem applied to $\tPoly{\ell}(T;\Symm)$ equipped with the standard $L^2$-product, along with the decomposition \eqref{eq:PolySym=Holy.oplus.cHoly} of this space and the compatibility condition expressed by the fact that both sides of \eqref{eq:P.Sigma.T} vanish for $q\in\Poly{1}(T)$ and $\btens{\upsilon}=\btens{0}$ (use \eqref{eq:DDT} and $\ell-1\ge 1$ to see that the right-hand side vanishes).

Using the commutation property \eqref{eq:DDT.commutation}, it can be checked that the following polynomial consistency properties hold:
For all $T\in\Th$, denoting by $\utens{I}_{\btens{\Sigma},E}^\ell$ the restriction of $\utens{I}_{\btens{\Sigma},T}^\ell$ to $E\in\ET$,
\begin{alignat}{6}
\label{eq:PE.poly.consistency}
\PSigmaE{\ell}(\ISigma[E]{\ell}\btens{\tau}_{|E})& = \btens{\tau}_{|E}\normal_E\cdot\normal_E
&\qquad&
\forall \btens{\tau}\in\tPoly{\ell}(T;\Symm)\,,&\quad&\forall E\in\ET,\\
\label{eq:PT.poly.consistency}
\PSigmaT{\ell}(\ISigma{\ell}\btens{\tau}) &= \btens{\tau}
&\qquad&
\forall \btens{\tau}\in\tPoly{\ell}(T;\Symm).
\end{alignat}

\subsection{Discrete complex}

Let $k\ge 3$.
The discrete version of the complex \eqref{eq:continuous.complex} is obtained arranging the spaces \eqref{eq:local.spaces} (with $\ell=k-1$) in a sequence and connecting them with the discrete operators:
\begin{equation}\label{eq:discrete.complex}
  \begin{tikzcd}
    \RT{1}
    \arrow{r}[above=2pt]{\IV{k}} & \uvec{V}_T^k
    \arrow{r}[above=2pt]{\uCsym{k-1}} & \utens{\Sigma}_T^{k-1}
    \arrow{r}[above=2pt]{\DD{k-2}} & \Poly{k-2}(T)
    \arrow{r}[above=2pt]{0} & 0.
  \end{tikzcd}
\end{equation}
The following remark will play a crucial role in establishing key properties of the discrete complex.
\begin{remark}[Hermite lifting on a subtriangulation]\label{rem:hermite-lifting}
  An element $\uvec{v}_T=\big(
      \bvec{v}_T,
      (\bvec{v}_E)_{E\in\ET},
      (\bvec{v}_V, \btens{G}_{\bvec{v},V})_{V\in\VT}
      \big)\in\uvec{V}_T^k$ can be lifted into a function over $T$ as described hereafter.
  Denote by $\mathfrak{T}_T$ a matching simplicial subtriangulation of $T$ the trace of which on $\partial T$ coincides with $\ET$, and let $\bvec{\mathfrak{H}}_T^k$ be the vector-valued Hermite finite element space of degree $k$ built on $\mathfrak{T}_T$ (see, e.g., \cite[Section 3.2]{Brenner.Scott:08} for the real-valued Hermite space).
  We take $\tilde{\bvec{v}}$ as the element of $\bvec{\mathfrak{H}}_T^k$ defined setting the degrees of freedom on $\partial T$ equal to $\big((\bvec{v}_E)_{E\in\ET},(\bvec{v}_V,\btens{G}_{\bvec{v},V})_{V\in\VT}\big)$ and interpolating the internal degrees of freedom (which determine in particular the projection on $\vPoly{k-3}(\mathfrak{t})$ for each triangle $\mathfrak{t}\in\mathfrak{T}_T$) from $\bvec{v}_T$.
  It follows that 
  \begin{equation}\label{eq:prop.lifting}
    \text{
      $\tilde{\bvec{v}}_{|\partial T} = \bvec{v}_{\ET}$,
      $\bvec{\pi}_{\cvec{P},T}^{k-3}\tilde{\bvec{v}} = \bvec{\pi}_{\cvec{P},T}^{k-3}\bvec{v}_T$,
      and $\mathbb{C}\btens{G}_{\bvec{v},V}=\mathbb{C}\GRAD\tilde{\bvec{v}}(\bvec{x}_V)=\SYM\CURL\tilde{\bvec{v}}(\bvec{x}_V)$ for all $V\in\VT$.
    }
  \end{equation}
  We note that the following integration-by-parts holds:
  For all $q\in C^1(\overline{T})$,
  \begin{multline}\label{eq:hermite-lifting:ibp}
    \int_T\SYM\CURL\tilde{\bvec{v}}:\HESS q
    - \sum_{E\in\ET}\omega_{TE}\sum_{V\in\VE}\omega_{EV}\SYM\CURL\tilde{\bvec{v}}(\bvec{x}_V)\normal_E\cdot\tangent_E~q(\bvec{x}_V)
    \\
    - \sum_{E\in\ET}\omega_{TE}\left[
      \int_E(\partial_{\tangent_E}\tilde{\bvec{v}}_{|E}\cdot\normal_E)\,\partial_{\normal_E} q
      - \int_E(\partial_{\tangent_E}^2\tilde{\bvec{v}}_{|E}\cdot\tangent_E)\, q
      \right] = 0.
  \end{multline}
  To establish this relation, it suffices to apply \eqref{eq:ibp.SigmaT} on each triangle $\mathfrak{t}\in\mathfrak{T}_T$ with $\bvec{\tau}=\SYM\CURL\tilde{\bvec{v}}_{|\mathfrak{t}}$, use $\DIV\DIV(\SYM\CURL\tilde{\bvec{v}}_{|\mathfrak{t}}) = 0$ and \cite[Lemma 2.2]{Chen.Huang:20} to transform the boundary terms into tangential derivatives of $\tilde{\bvec{v}}_{|\mathfrak{t}}$, sum up the resulting formulas over $\mathfrak{t}\in\mathfrak{T}_T$, and notice that the terms on the internal boundaries $\cup_{\mathfrak{t}\in\mathfrak{T}_T}\partial \mathfrak{t}\cap T$ cancel out since the tangential derivatives on the triangle edges and the nodal values of the gradients of $\tilde{\bvec{v}}$ are continuous across the triangles.
  Additionally, using the first two relations in \eqref{eq:prop.lifting} in \eqref{eq:full.CsymT} with $\btens{\tau}\in\tPoly{k-2}(T;\Symm)$ and using the integration by parts \eqref{eq:ibp.VT} (which is valid since $\tilde{\bvec{v}}\in \bvec{H}^1(T;\Real^2)$), it holds
  \begin{equation}\label{eq:hermite-lifting:curl}
    \btens{\pi}^{k-2}_{\ctens{P},T}(\Csym{k-1}\uvec{v}_T)
    = \btens{\pi}^{k-2}_{\ctens{P},T}(\SYM\CURL\tilde{\bvec{v}}).
  \end{equation}
\end{remark}

\begin{theorem}[Exactness]\label{thm:exactness}
  With $T$ contractible polygon, the complex \eqref{eq:discrete.complex} is exact.
\end{theorem}

\begin{proof}
  We have to prove the following relations:
  \begin{align}
    \IV{k}\RT{1} &= \Ker\uCsym{k-1}, \label{eq:IVT.RT1=Ker.CsymT} \\
    \Image \DD{k-2} &= \Poly{k-2}(T), \label{eq:Im.DDT=P1} \\
    \Image\uCsym{k-1} &= \Ker \DD{k-2}. \label{eq:Im.CsymT=Ker.DDT}
  \end{align}
  \underline{1. \emph{Proof of \eqref{eq:IVT.RT1=Ker.CsymT}.}}
  We start by proving that
  \begin{equation}\label{eq:IVT.RT1.subset.Ker.CsymT}
    \IV{k}\RT{1} \subset \Ker\uCsym{k-1}.
  \end{equation}
  To this end, we take a generic $\bvec{w}\in\RT{1}$, set $\hat{\uvec{w}}_T\coloneq\IV{k}\bvec{w}$, and show that $\uCsym{k-1}\hat{\uvec{w}}_T = \utens{0}$.
  By the polynomial consistency property \eqref{eq:CsymT:polynomial.consistency} along with $\RT{1}(T)\subset\vPoly{k}(T;\Real^2)$, it holds $\Csym{k-1}\hat{\uvec{w}}_T = \SYM\CURL\bvec{w} = \bvec{0}$, where the conclusion follows using the leftmost portion of the continuous complex \eqref{eq:continuous.complex}.
  As a consequence, the internal components of $\uCsym{k-1}\hat{\uvec{w}}_T$ vanish.
  
  Let us now consider the boundary components.
  We start by noticing that, by polynomial consistency of the construction \eqref{eq:v.ET}, $\hat{\bvec{w}}_{\ET} = \bvec{w}_{|\partial T}$.
  Moreover, for all $E\in\ET$, $\bvec{w}_{|E}\cdot\normal_E\in\Poly{0}(E)$ (see \cite[Proposition 8]{Di-Pietro.Droniou:21*1}), so that $\partial_{\tangent_E}\hat{\bvec{w}}_{\ET}\cdot\normal_E = \partial_{\tangent_E}\bvec{w}_{|E}\cdot\normal_E = 0$, i.e., the first edge component of $\uCsym{k-1}\hat{\uvec{w}}_T$ vanishes.
  Additionally, since $\hat{\bvec{w}}_{\ET}\cdot\tangent_E = \bvec{w}_{|E}\cdot\tangent_E\in\Poly{1}(E)$, the function $\partial_{\tangent_E}^2\hat{\bvec{w}}_{\ET}\cdot\tangent_E$ also vanishes.
  Finally, vertex components vanish since $\mathbb{C}\GRAD\bvec{w}(\bvec{x}_V) = \SYM\CURL\bvec{w}(\bvec{x}_V) = \btens{0}$ by \eqref{eq:continuous.complex}, thus proving \eqref{eq:IVT.RT1.subset.Ker.CsymT}
  \medskip
  
  We next show that
  \begin{equation}\label{eq:Ker.CsymT.subset.IVT.RT1}
    \Ker\uCsym{k-1} \subset \IV{k}\RT{1},
  \end{equation}
  i.e., for all $\uvec{v}_T\in\uvec{V}_T^k$ such that
  \begin{equation}\label{eq:CsymT=0}
    \uCsym{k-1}\uvec{v}_T = \utens{0},
  \end{equation}
  there exists $\bvec{w}\in\RT{1}(T)$ such that $\uvec{v}_T = \IV{k}\bvec{w}$.
  We start by examining the boundary function $\bvec{v}_{\ET}$ defined by \eqref{eq:v.ET}.
  For any $V\in\VT$, the condition $\mathbb{C}\btens{G}_{\bvec{v},V}=\btens{0}$ resulting from \eqref{eq:CsymT=0} and the definition \eqref{eq:def.tensor.C} of $\mathbb{C}$ imply the existence of $(\beta_V)_{V\in\VT}\in\Real^{\card(\VT)}$ such that
  \begin{equation}\label{eq:CsymT=0.V}
    \btens{G}_{\bvec{v},V} = \beta_V \btens{I}_2\qquad\forall V\in\VT,
  \end{equation}
  with $\btens{I}_2$ denoting as before the $2\times 2$ identity matrix.
  Let now $E\in\ET$ and denote by $V_1,\,V_2$ its vertices numbered so that $\tangent_E = \bvec{x}_{V_2} - \bvec{x}_{V_1}$.
  The conditions relative to $E$ resulting from \eqref{eq:CsymT=0} translate as follows:
  \begin{gather}\label{eq:CsymT=0.E:3}
    \pi_{\mathcal{P},E}^{k-3}\left(\partial_{\tangent_E}(\bvec{v}_{\ET}\cdot\normal_E)\right)
    = 0,
    \\ \label{eq:CsymT=0.E:2}
    \partial_{\tangent_E}^2(\bvec{v}_{\ET}\cdot\tangent_E)
    = 0
    \implies
    \bvec{v}_{\ET}\cdot\tangent_E\in\Poly{1}(E).
  \end{gather}
  Combining \eqref{eq:CsymT=0.E:2} (which implies, in particular, that $\partial_{\tangent_E}(\bvec{v}_{\ET}\cdot\tangent_E)$ is constant along $E$) with the definition \eqref{eq:v.ET} of $\bvec{v}_{\ET}$, we obtain
  \[
  \btens{G}_{\bvec{v},V_1}\tangent_E\cdot \tangent_E=\partial_{\tangent_E}(\bvec{v}_{\ET}\cdot\tangent_E)_{|E}(\bvec{x}_{V_1})=\partial_{\tangent_E}(\bvec{v}_{\ET}\cdot\tangent_E)_{|E}(\bvec{x}_{V_2})=\btens{G}_{\bvec{v},V_2}\tangent_E\cdot \tangent_E.
  \]
  Owing to \eqref{eq:CsymT=0.V} this yields $\beta_{V_1}=\beta_{V_2}$. 
  Repeating this reasoning for all edges $E\in\ET$, we infer the existence of $\beta\in\Real$ such that
  \begin{equation}\label{eq:betaV=beta}
    \text{
      $\beta_V = \beta$ and
      $\partial_{\tangent_E}(\bvec{v}_{\ET})_{|E}(\bvec{x}_V)=\beta\tangent_E$ for all $V\in\VT$.
    }
  \end{equation}
  
  We now exhibit a function $\bvec{w}\in\RT{1}(T)$ such that, for all $E\in\ET$,
  \begin{equation}\label{eq:conditions.RT1}
    \text{      
      $\bvec{w}_{|E}\cdot\tangent_E = \bvec{v}_{\ET}\cdot\tangent_E$ and
      $\bvec{w}(\bvec{x}_V)\cdot\normal_E = \bvec{v}_{\ET}(\bvec{x}_V)\cdot\normal_E$ for all $\bvec{x}_V\in\VE$.
    }
  \end{equation}
  We take $\bvec{w}$ of the form $\bvec{w}(\bvec{x}) = \bvec{\alpha} + \beta\bvec{x}$ with $\bvec{\alpha}\in\Real^2$ such that, for a fixed vertex $V_0\in\VT$, $\bvec{w}(\bvec{x}_{V_0}) = \bvec{v}_{V_0}$ and $\beta$ as in \eqref{eq:betaV=beta}.
  Let us number the remaining vertices coherently with the orientation of $T$, and denote by $E_i$ the edge of endpoints $\bvec{x}_{V_i},\,\bvec{x}_{V_{i+1}}$, $i\ge 0$.
  Set $i=0$.
  Since $\bvec{w}_{|E_i}\cdot\normal_{E_i}$ is constant along $E_i$ by \cite[Proposition 8]{Di-Pietro.Droniou:21*1}, $\bvec{w}(\bvec{x}_{V_{i+1}})\cdot\normal_{E_i} = \bvec{w}(\bvec{x}_{V_i})\cdot\normal_{E_i} = \bvec{v}_{V_i}\cdot\normal_{E_i}$.
  Moreover, for any $\bvec{x}\in E_i$,
  \[
  \bvec{w}(\bvec{x})\cdot\tangent_{E_i}
  = \left[
    \bvec{w}(\bvec{x}_{V_i}) + \beta(\bvec{x} - \bvec{x}_{V_i})
    \right]\cdot\tangent_{E_i}
  = \left[
    \bvec{v}_{V_i} + \beta(\bvec{x} - \bvec{x}_{V_i})
    \right]\cdot\tangent_{E_i}
  = \bvec{v}_{\ET}(\bvec{x})\cdot\tangent_E,  
  \]
  the conclusion following by \eqref{eq:CsymT=0.E:2} combined with \eqref{eq:betaV=beta}.
  These conditions are precisely \eqref{eq:conditions.RT1} for $E_i = E_0$.
  Iterating this reasoning for $i=1,\ldots,\card(\ET)-1$ allows one to prove \eqref{eq:conditions.RT1} for all $E\in\ET$.
  \smallskip
  
  We next prove that
  \begin{equation}\label{eq:w.pT=vET}
    \text{for all $E\in\ET$, $\bvec{w}_{|E} = \bvec{v}_{\ET}$ on $E$.}
  \end{equation}
  By \eqref{eq:conditions.RT1}, the tangential component of $\bvec{w}$ coincides with that of $\bvec{v}_{\ET}$ on every edge.
  It only remains to prove the equality of their normal components.
  Let $E\in\ET$ be the edge of vertices $V_1,\,V_2$ and set, for the sake of brevity, $\phi_E\coloneq\bvec{v}_{\ET|E}\cdot\normal_E\in\Poly{k}(E)$.
  Combining \eqref{eq:betaV=beta} and \eqref{eq:CsymT=0.E:3}, we infer
  \begin{gather*}    
    \text{
      $\partial_{\tangent_E}\phi_E(\bvec{x}_{V_i})
      = \partial_{\tangent_E}(\bvec{v}_{\ET})_{|E}(\bvec{x}_{V_i})\cdot\normal_E
      = \beta\tangent_E\cdot\normal_E = 0$ for all $i\in\{1,2\}$
      and $\pi_{\mathcal{P},E}^{k-3}(\partial_{\tangent_E}\phi_E) = 0$.
    }
  \end{gather*}
  It is easy to check that these conditions along with $\partial_{\tangent_E}\phi_E\in\Poly{k-1}(E)$ enforce $\partial_{\tangent_E}\phi_E=0$, so that $\phi_E\in\Poly{0}(E)$ is identically equal to $\bvec{w}_{|E}\cdot\normal_E$ on $E$ by the second condition in \eqref{eq:conditions.RT1}, thus concluding the proof of \eqref{eq:w.pT=vET}.
  \smallskip
  
  Respectively denoting by $\uvec{v}_{\partial T}$ and $\IV[\partial T]{k}$ the restrictions to $\partial T$ of $\uvec{v}_T$ and $\IV{k}$ (the latter being obtained collecting only the boundary components of \eqref{eq:IVT}), \eqref{eq:w.pT=vET} and \eqref{eq:CsymT=0.V} imply
  \begin{equation}\label{eq:vpT=IVT.w}
    \uvec{v}_{\partial T} = \IV[\partial T]{k}\bvec{w}_{|\partial T}.
  \end{equation}  
  To conclude the proof of \eqref{eq:Ker.CsymT.subset.IVT.RT1}, it suffices to show that $\bvec{v}_T = \bvec{\pi}_{\cvec{P},T}^{k-2}\bvec{w}=\bvec{w}$ (notice that $k-2\ge 1$ since $k\ge 3$).
  The condition $\btens{\pi}_{\ctens{H},T}^{\compl,k-1}\big(\Csym{k-1}\uvec{v}_T\big) = \btens{0}$ resulting from \eqref{eq:CsymT=0} implies, accounting for \eqref{eq:full.CsymT}:
  For all $\btens{\tau}\in\cHoly{k-1}(T)$,
  \[
  \int_T\bvec{v}_T\cdot\VROT\btens{\tau}
  = \sum_{E\in\ET}\omega_{TE}\int_E\bvec{v}_{\ET}\cdot(\btens{\tau}\,\tangent_E)
  = \sum_{E\in\ET}\omega_{TE}\int_E\bvec{w}_{|E}\cdot(\btens{\tau}\,\tangent_E)
  = \int_T \bvec{w}\cdot\VROT\btens{\tau},
  \]
  where we have used \eqref{eq:w.pT=vET} in the second equality and the integration by parts formula \eqref{eq:ibp.VT} to conclude.
  Since $\VROT:\cHoly{k-1}(T)\to\vPoly{k-2}(T;\Real^2)$ is an isomorphism, this shows that $\bvec{v}_T = \bvec{w}$, hence, recalling \eqref{eq:vpT=IVT.w}, $\uvec{v}_T = \IV{k}\bvec{w}$, proving \eqref{eq:Ker.CsymT.subset.IVT.RT1}.
  \medskip\\
  \underline{2. \emph{Proof of \eqref{eq:Im.DDT=P1}.}}
  Given $q\in\Poly{k-2}(T)\subset L^2(T)$, by surjectivity of $\DIV\VDIV$ (see \cite{Chen.Huang:18}) there exists $\btens{\tau}_q\in \btens{H}^2(T;\Symm)$ such that $\DIV\VDIV\btens{\tau}_q = q$.
  The surjectivity of $\DD{k-2}$ is then proved observing that, by \eqref{eq:DDT.commutation},
  \[
  \DD{k-2}\big(\ISigma{k-1}\btens{\tau}_q\big)
  = \pi_{\mathcal{P},T}^{k-2}\big(\DIV\VDIV\btens{\tau}_q\big)
  = \pi_{\mathcal{P},T}^{k-2} q
  = q.
  \]
  \medskip\\
  \underline{3. \emph{Proof of \eqref{eq:Im.CsymT=Ker.DDT}.}}
  We start by proving that
  \begin{equation}\label{eq:Im.CsymT.subset.Ker.DDT}
    \Image\uCsym{k-1} \subset \Ker \DD{k-2}.
  \end{equation}
  To this end, let $\uvec{v}_T\in\uvec{V}_T^k$.
  Writing the definition \eqref{eq:DDT} of $\DD{k-2}$ with $\utens{\tau}_T = \uCsym{k-1}\uvec{v}_T$ and recalling the definitions \eqref{eq:uCsymT} of $\uCsym{k-1}$ and \eqref{eq:v.ET} of $\bvec{v}_{\ET}$, we have, for all $q\in\Poly{k-2}(T)$,
  \begin{equation}\label{cestdicikonpar}
  \begin{aligned}
    \int_T \DD{k-2}{}&\big(\uCsym{k-1}\uvec{v}_T\big)~q\\
    &=
    \int_T \cancelto{\btens{\pi}_{\cvec{P},T}^{k-2}}{\btens{\pi}_{\ctens{H},T}^{k-4}}\big(\Csym{k-1}\uvec{v}_T\big):\HESS q
    -\sum_{E\in\ET}\omega_{TE}\sum_{V\in\VE}\omega_{EV}(\mathbb{C}\btens{G}_{\bvec{v},V})\normal_E\cdot\tangent_E~q(\bvec{x}_V)
    \\
    &\quad
    -\sum_{E\in\ET}\omega_{TE}\left(
    \int_E\cancel{\pi_{\mathcal{P},E}^{k-3}}(\partial_{\tangent_E}\bvec{v}_{\ET}\cdot\normal_E)~\partial_{\normal_E}q
    -\int_E \partial_{\tangent_E}^2\bvec{v}_{\ET}\cdot\tangent_E~q
    \right),
  \end{aligned}
  \end{equation}
  where the substitution of $\btens{\pi}_{\ctens{H},T}^{k-4}$ with $\btens{\pi}_{\cvec{P},T}^{k-2}$ is made possible by the fact that $\HESS q\in\Holy{k-4}(T)\subset\tPoly{k-2}(T)$,
  while the removal of $\pi_{\mathcal{P},E}^{k-3}$ is a consequence of $\partial_{\normal_E}q_{|E}\in\Poly{k-3}(E)$.
  Denote by $\tilde{\bvec{v}}\in\bvec{H}^2(T;\Real^2)$ the lifting of $\uvec{v}_T$ in a Hermite finite element space on a subtriangulation of $T$ obtained as in Remark \ref{rem:hermite-lifting}.
  Recalling \eqref{eq:hermite-lifting:curl} and \eqref{eq:prop.lifting}, \eqref{cestdicikonpar} becomes
  \[
  \begin{aligned}
    \int_T \DD{k-2}{}\big(\uCsym{k-1}\uvec{v}_T\big)~q
    &=
    \int_T\SYM\CURL\tilde{\bvec{v}}:\HESS q
    -\!\!\sum_{E\in\ET}\!\omega_{TE}\!\!\sum_{V\in\VE}\!\omega_{EV}\SYM\CURL\tilde{\bvec{v}}(\bvec{x}_V)\normal_E\cdot\tangent_E~q(\bvec{x}_V)
    \\
    &\quad
    -\sum_{E\in\ET}\omega_{TE}\left(
    \int_E(\partial_{\tangent_E}\tilde{\bvec{v}}_{|E}\cdot\normal_E)~\partial_{\normal_E}q
    -\int_E \partial_{\tangent_E}^2\tilde{\bvec{v}}_{|E}\cdot\tangent_E~q
    \right)=0,
  \end{aligned}
  \]
  the conclusion being a consequence of \eqref{eq:hermite-lifting:ibp}.
  This proves \eqref{eq:Im.CsymT.subset.Ker.DDT}.
  \smallskip

  To prove \eqref{eq:Im.CsymT=Ker.DDT}, it only remains to show that
  \begin{equation}\label{eq:dim.Im.CsymT=dim.Ker.DDT}
    \DIM(\Image\uCsym{k-1}) = \DIM(\Ker \DD{k-2}).
  \end{equation}
  By the rank-nullity theorem along with \eqref{eq:IVT.RT1=Ker.CsymT} and \eqref{eq:dim.VT},
  \[
  \DIM(\Image\uCsym{k-1})
  = \DIM\uvec{V}_T^k - \DIM(\RT{1}(T))
  = k(k-1) + 2k\card(\VT) - 3.
  \]
  On the other hand, again by the rank-nullity theorem along with \eqref{eq:Im.DDT=P1} and \eqref{eq:dim.SigmaT} written for $\ell=k-1$,
  \[
  \begin{aligned}
    \DIM(\Ker \DD{k-2})
    &= \DIM\utens{\Sigma}_T^{k-1} - \DIM(\Poly{k-2}(T))
    \\
    &= \tfrac32 k(k-1) + 2k\card(\VT) - 3 - \tfrac12k(k-1)
    \\
    &= k(k-1) + 2k\card(\VT) - 3.
  \end{aligned}
  \]
  This proves \eqref{eq:dim.Im.CsymT=dim.Ker.DDT} and concludes the proof of \eqref{eq:Im.CsymT=Ker.DDT}.
\end{proof}

We close this section by establishing a link between the tensor potential reconstruction in $\utens{\Sigma}_T^{k-1}$ and the full and discrete curls on $\uvec{V}_T^k$.
As pointed out in \cite{Di-Pietro.Droniou:21*1}, this type of results play a key role in the proof of optimal consistency properties for discrete operators, and thus rates of convergence for the related schemes.
This is the reason why, even though Lemma \ref{lem:PSigmaT.uCsymT=CsymT} is not used in the numerical analysis of the Kirchhoff--Love model considered in Section \ref{sec:application} (as this model does not involve the symmetric curl), we state and prove this link potential--curls link.

\begin{lemma}[Tensor potential reconstruction and discrete curls]\label{lem:PSigmaT.uCsymT=CsymT}
  It holds, for all $\uvec{v}_T\in\uvec{V}_T^k$,
  \[
  \PSigmaT{k-1}\big(
  \uCsym{k-1}\uvec{v}_T
  \big) = \Csym{k-1}\uvec{v}_T.
  \]
\end{lemma}

\begin{proof}
  Let $\uvec{v}_T\in\uvec{V}_T^k$.
  We start by noticing that, for all $E\in\ET$, $\PSigmaE{k-1}\big(\uCsym[E]{k-1}\uvec{v}_E\big)$ (with $\uCsym[E]{k-1}$ and $\uvec{v}_E$ denoting the restrictions of $\uCsym{k-1}$ and $\uvec{v}_T$ to $E$) satisfies, by its definition \eqref{eq:P.Sigma.E} written for $\ell=k-1$,
  \begin{gather*}
  \text{
    $\PSigmaE{k-1}\big(\uCsym[E]{k-1}\uvec{v}_E\big)(\bvec{x}_V)
    = \mathbb{C}\btens{G}_{\bvec{v},V}\normal_E\cdot\normal_E
    = (\partial_{\tangent_E}\bvec{v}_{\ET}\cdot\normal_E)(\bvec{x}_V)$ for all $V\in\VE$,
  }
  \\
  \pi_{\mathcal{P},E}^{k-3}\left[
    \PSigmaE{k-1}\big(\uCsym[E]{k-1}\uvec{v}_E\big)
    \right]
  = \pi_{\mathcal{P},E}^{k-3}(\partial_{\tangent_E}\bvec{v}_{\ET}\cdot\normal_E),
  \end{gather*}
  where we have used \cite[Eq.~(3)]{Chen.Huang:20} and introduced the lifting $\tilde{\bvec{v}}$ of $\uvec{v}_T$ defined in Remark \ref{rem:hermite-lifting} to write, in the first line,
  \[
     \mathbb{C}\btens{G}_{\bvec{v},V}\normal_E\cdot \normal_E
     =(\SYM\CURL\tilde{\bvec{v}}(\bvec{x}_V))\normal_E \cdot\normal_E
     =(\partial_{\tangent_E}\bvec{\tilde{v}}(\bvec{x}_V))\cdot \normal_E
     =(\partial_{\tangent_E}\bvec{v}_{\ET}(\bvec{x}_V))\cdot \normal_E.   
  \]  
  Since $\partial_{\tangent_E}\bvec{v}_{\ET}\cdot\normal_E\in\Poly{k-1}(E)$, these conditions ensure that
  \[
  \PSigmaE{k-1}\big(\uCsym[E]{k-1}\uvec{v}_E\big) = \partial_{\tangent_E}\bvec{v}_{\ET}\cdot\normal_E.
  \]
  We next write the definition \eqref{eq:P.Sigma.T} of $\btens{P}_{\btens{\Sigma},T}^k$ with $\utens{\tau}_T = \uCsym{k-1}\uvec{v}_T$ and use the above relation along with \eqref{eq:Im.CsymT=Ker.DDT} to infer, for all $(q,\btens{\upsilon})\in\Poly{k+1}(T)\times\cHoly{k-1}(T)$,
  \[
  \begin{aligned}
    &\int_T\PSigmaT{k-1}\big(
    \uCsym{k-1}\uvec{v}_T
    \big):(\HESS q + \btens{\upsilon})
    \\
    &\quad= \sum_{E\in\ET}\omega_{TE}\sum_{V\in\VE}\omega_{EV}(\mathbb{C}\btens{G}_{\bvec{v},V}\normal_E\cdot\tangent_E\, q)(\bvec{x}_V)
    \\
    &\qquad + \sum_{E\in\ET}\omega_{TE}\left[
      \int_E (\partial_{\tangent_E}\bvec{v}_{\ET}\cdot\normal_E)\,\partial_{\normal_E} q
      - \int_E (\partial_{\tangent_E}^2\bvec{v}_{\ET}\cdot\tangent_E)\,q
      \right]
    + \int_T\cancel{\btens{\pi}_{\ctens{H},T}^{\compl,k-1}}(\Csym{k-1}\utens{\tau}_T):\btens{\upsilon},
  \end{aligned}
  \]
  where the cancellation of the projector is possible since $\btens{\upsilon}\in\cHoly{k-1}(T)$.
  Using the lifting $\tilde{\bvec{v}}$ of $\uvec{v}_T$ defined in Remark \ref{rem:hermite-lifting} and the integration by parts formula \eqref{eq:hermite-lifting:ibp}, we can go on writing
  \begin{equation}\label{eq:PSigmaT.uCsymT=CsymT:1}
    \int_T\PSigmaT{k-1}\big(
    \uCsym{k-1}\uvec{v}_T
    \big):(\HESS q + \btens{\upsilon})
    =
    \int_T \SYM\CURL\tilde{\bvec{v}}:\HESS q
    + \int_T\Csym{k-1}\utens{\tau}_T:\btens{\upsilon}.
  \end{equation}
  By the integration by parts formula \eqref{eq:ibp.VT} with $(\bvec{v},\btens{\tau}) = (\tilde{\bvec{v}},\HESS q)$ (which is justified since $\tilde{\bvec{v}}\in \bvec{H}^1(T;\Real^2)$), we have
  \begin{equation}\label{eq:PSigmaT.uCsymT=CsymT:2}
    \begin{aligned}
    \int_T \SYM\CURL\tilde{\bvec{v}}:\HESS q
    &= -\int_T\tilde{\bvec{v}}\cdot\cancel{\VROT(\HESS q)}
    + \sum_{E\in\ET}\omega_{TE}\int_{E}\HESS q\,\tangent_E\cdot\tilde{\bvec{v}}
    \\
    &= \sum_{E\in\ET}\omega_{TE}\int_{E}\HESS q\,\tangent_E\cdot\bvec{v}_{\ET}
    =\int_T\Csym{k-1}\uvec{v}_T:\HESS q,
    \end{aligned}
  \end{equation}
  where we have used the fact that $\tilde{\bvec{v}}_{|\partial T} = \bvec{v}_{\ET}$ to pass to the second line and the definition \eqref{eq:full.CsymT} of $\Csym{k-1}$ with $\btens{\tau}=\HESS q$ to conclude.
  Plugging \eqref{eq:PSigmaT.uCsymT=CsymT:2} into \eqref{eq:PSigmaT.uCsymT=CsymT:1}, the conclusion follows observing that $\HESS q + \btens{\upsilon}$ spans $\tPoly{k-1}(T;\Symm)$ as $(q,\btens{\upsilon})$ spans $\Poly{k+1}(T)\times\cHoly{k-1}(T)$ by \eqref{eq:PolySym=Holy.oplus.cHoly} written for $m = k-1$.
\end{proof}


\section{Application to Kirchhoff--Love plates}\label{sec:application}

We apply the above developments to the design of a polygonal method for problem \eqref{eq:strong.problem}.
We assume that $D>0$, $\nu\in (0,1)$, and we note that the tensor $\mathbb{A}$ is invertible, with
\[
\mathbb{A}^{-1}\btens{\tau}=\frac{1}{D(1-\nu)}\left(\btens{\tau} - \frac{\nu}{1+\nu} \tr(\btens{\tau})\btens{I}\right)\qquad\forall \btens{\tau}\in\Symm.
\]
In particular, the following coercivity and bound estimates hold, in which $|{\cdot}|$ denotes the Frobenius norm:
\begin{equation}\label{eq:A.coer.bound}
\mathbb{A}^{-1}\btens{\tau}:\btens{\tau}\ge \frac{|\btens{\tau}|^2}{D(1+\nu)}\quad\mbox{ and }\quad
|\mathbb{A}^{-1}\btens{\tau}|\le \frac{2}{D(1-\nu)}|\btens{\tau}|,\qquad\forall\btens{\tau}\in\Symm.
\end{equation}
Assuming $f\in L^2(\Omega)$, we thus consider the following weak formulation of problem \eqref{eq:strong.problem}: 
Find $(\btens{\sigma},u)\in\Hdivdiv{\Omega}{\Symm}\times L^2(\Omega)$ such that
\begin{equation}\label{eq:continuous.problem}
\begin{alignedat}{2}
  \int_\Omega\mathbb{A}^{-1}\btens{\sigma}:\btens{\tau} + \int_\Omega\DIV\VDIV\btens{\tau}\, u &= 0 &\qquad&\forall\btens{\tau}\in\Hdivdiv{\Omega}{\Symm},
  \\
  -\int_\Omega\DIV\VDIV\btens{\sigma}\, v &= \int_\Omega fv &\qquad& \forall v\in L^2(\Omega).
\end{alignedat}
\end{equation}

\subsection{Global discrete spaces}

Let $\ell\ge 2$ and denote by $\Mh = (\Th,\Eh)$ a polygonal mesh in the sense of \cite[Definition 1.4]{Di-Pietro.Droniou:20}, with $\Th$ collecting the polygonal mesh elements and $\Eh$ the mesh faces.
We additionally denote by $\Vh$ the set collecting the edge endpoints.
The global discrete counterpart of $\Hdivdiv{\Omega}$ on $\Mh$ is defined glueing the boundary components of the local spaces \eqref{eq:SigmaT} defined on each mesh element $T\in\Th$:
\begin{align*}
  \utens{\Sigma}_h^\ell&\coloneq\Big\{
  \begin{aligned}[t]
    &\utens{\tau}_h
    =\big(
    (\btens{\tau}_{\cvec{H},T},\btens{\tau}_{\cvec{H},T}^\compl)_{T\in\Th},
    (\tau_E,D_{\bvec{\tau},E})_{E\in\Eh}
    (\btens{\tau}_V)_{V\in\Vh}
    \big)\st
    \\
    &\qquad\text{
      $(\btens{\tau}_{\cvec{H},T},\btens{\tau}_{\cvec{H},T}^\compl)\in\Holy{\ell-3}(T)\times\cHoly{\ell}(T)$ for all $T\in\Th$,}\\
    &\qquad\text{
      $\tau_E\in\Poly{\ell-2}(E)$ and $D_{\bvec{\tau},E}\in\Poly{\ell-1}(E)$ for all $E\in\Eh$,
    }
    \\
    &\qquad\text{
      $\btens{\tau}_V\in\Symm$ for all $V\in\Vh$
    }
    \Big\}.
  \end{aligned}
\end{align*}
Correspondingly, the global interpolator on $\utens{\Sigma}_h^\ell$ is defined patching together the local interpolators \eqref{eq:ISigmaT}:
For all $\btens{\tau}\in\btens{H}^2(\Omega;\Symm)$,
\[
  \utens{I}^\ell_{\btens{\Sigma},h}\btens{\tau}\coloneq
  \Big(
  \begin{aligned}[t]
  &(\btens{\pi}_{\cvec{H},T}^{\ell-3}\btens{\tau},
    \btens{\pi}_{\cvec{H},T}^{\compl,\ell}\btens{\tau})_{T\in\Th},
    \\
  &\big(
  \pi_{\cvec{P},E}^{\ell-2}(\btens{\tau}_{|E}\normal_E\cdot\normal_E),
  \pi_{\cvec{P},E}^{\ell-1}\big[
  \partial_{\tangent_E}(\btens{\tau}_{|E}\normal_E\cdot\tangent_E)
  + (\VDIV\btens{\tau})_{|E}\cdot\normal_E
  \big]
  \big)_{E\in\Eh},
  \\
  &\big(\btens{\tau}(\bvec{x}_V)\big)_{V\in\Vh}
  \Big).
  \end{aligned}
\]

We additionally define, for any integer $m\ge 0$, the space of broken polynomials on $\Th$ of total degree $\le m$ as
\[
\Poly{m}(\Th)
\coloneq\left\{
q_h\in L^2(\Omega)\st\text{$q_{|T}\in\Poly{m}(T)$ for all $T\in\Th$}
\right\}.
\]
Vector and tensor versions of this space are denoted in boldface and we specify the codomain to distinguish them.
Interpolation on $\Poly{m}(\Th)$ is done through the $L^2$-orthogonal projector $\pi_{\mathcal{P},h}^m$.

\subsection{Global discrete $L^2$-product and norm of moment tensors}

We define the following discrete $L^2$-product on $\utens{\Sigma}_h^\ell$:
\[
(\utens{\upsilon}_h,\utens{\tau}_h)_{\btens{\Sigma},h}
\coloneq
\sum_{T\in\Th}\left(
\int_T\PSigmaT{\ell}\utens{\upsilon}_T:\PSigmaT{\ell}\utens{\tau}_T
+ s_{\btens{\Sigma},T}(\utens{\upsilon}_T,\utens{\tau}_T)
\right)
\qquad\forall (\utens{\upsilon}_h,\utens{\tau}_h)\in\utens{\Sigma}_h^\ell\times\utens{\Sigma}_h^\ell,
\]
with stabilization bilinear form
\[
\begin{aligned}
  s_{\btens{\Sigma},T}(\utens{\upsilon}_T,\utens{\tau}_T)
  &\coloneq\sum_{E\in\ET}h_T\int_E
  (\PSigmaT{\ell}\utens{\upsilon}_T\normal_E\cdot\normal_E - \PSigmaE{\ell}\utens{\upsilon}_E)
  (\PSigmaT{\ell}\utens{\tau}_T\normal_E\cdot\normal_E - \PSigmaE{\ell}\utens{\tau}_E)
  \\
  &\quad
  + \sum_{E\in\ET}h_T^3\int_E\Big[
    \left(
    \partial_{\tangent_E}(\PSigmaT{\ell}\utens{\upsilon}_T\normal_E\cdot\tangent_E) + \VDIV\PSigmaT{\ell}\utens{\upsilon}_T\cdot\normal_E - D_{\btens{\upsilon},E}
    \right)
    \\
    &\hphantom{+ \sum_{E\in\ET}h_T^3\int_E\Big[}\qquad
      \times\left(
      \partial_{\tangent_E}(\PSigmaT{\ell}\utens{\tau}_T\normal_E\cdot\tangent_E) + \VDIV\PSigmaT{\ell}\utens{\tau}_T\cdot\normal_E - D_{\btens{\tau},E}
      \right)\Big]
    \\
    &\quad
    + \sum_{V\in\VT} h_T^2\,
    (\PSigmaT{\ell}\utens{\upsilon}_T(\bvec{x}_V) - \btens{\upsilon}_V):
    (\PSigmaT{\ell}\utens{\tau}_T(\bvec{x}_V) - \btens{\tau}_V).
\end{aligned}
\]
The corresponding \emph{operator norm} is
\[
\norm[\btens{\Sigma},h]{\utens{\tau}_h}\coloneq (\utens{\tau}_h,\utens{\tau}_h)_{\btens{\Sigma},h}^{\frac12}.
\]
We denote by $\norm[\btens{\Sigma},T]{{\cdot}}$ the norm on $\utens{\Sigma}_T^\ell$ obtained restricting the above norm to the element $T\in\Th$.

\subsection{Discrete problem and main results}
The discrete problem reads:
Find $(\utens{\sigma}_h,u_h)\in\utens{\Sigma}_h^\ell\times \Poly{\ell-1}(\Th)$ such that
\begin{subequations}\label{eq:discrete}
  \begin{alignat}{4}\label{eq:discrete:flux}
    a_h(\utens{\sigma}_h,\utens{\tau}_h) + b_h(\utens{\tau}_h,u_h) &= 0 &\qquad&\forall\utens{\tau}_h\in\utens{\Sigma}_h^\ell,
    \\ \label{eq:discrete:balance}
    -b_h(\utens{\sigma}_h,v_h) &= \int_\Omega f v_h &\qquad&\forall v_h\in \Poly{\ell-1}(\Th),
  \end{alignat}
\end{subequations}
with bilinear forms $a_h:\utens{\Sigma}_h^\ell\times\utens{\Sigma}_h^\ell\to\Real$
and $b_h:\utens{\Sigma}_h^\ell\times \Poly{\ell-1}(\Th)\to\Real$ such that
\begin{equation}\label{eq:ah.bh}
  \begin{aligned}
  a_h(\utens{\sigma}_h,\utens{\tau}_h)\coloneq{}&
  \sum_{T\in\Th}\left(
  \int_T \mathbb{A}^{-1}\PSigmaT{\ell}\utens{\sigma}_T:\PSigmaT{\ell}\utens{\tau}_T
  +\frac{1}{D(1+\nu)} s_{\btens{\Sigma},T}(\utens{\sigma}_T,\utens{\tau}_T)
  \right),\\
  b_h(\utens{\tau}_h,v_h)\coloneq{}&\sum_{T\in\Th}\int_T \DD{\ell-1}\utens{\tau}_T\,v_T,
  \end{aligned}
\end{equation}
where $\utens{\tau}_T\in\utens{\Sigma}_T^\ell$ and $v_T\in\Poly{\ell-1}(T)$ respectively denote the restrictions of $\utens{\tau}_h$ and $v_h$ to $T$.
Adding together both equations in \eqref{eq:discrete} we obtain the following equivalent variational formulation:
Find $(\utens{\sigma}_h,u_h)\in\utens{\Sigma}_h^\ell\times \Poly{\ell-1}(\Th)$ such that
\begin{equation}\label{eq:discrete:variational}
  A_h((\utens{\sigma}_h,u_h),(\utens{\tau}_h,v_h))
  = \int_\Omega fv_h\qquad
  \forall(\utens{\tau}_h,v_h)\in\utens{\Sigma}_h^\ell\times \Poly{\ell-1}(\Th),
\end{equation}
with bilinear form $A_h:(\utens{\Sigma}_h^\ell\times \Poly{\ell-1}(\Th))^2\to\Real$ such that
\[
A_h((\utens{\sigma}_h,u_h),(\utens{\tau}_h,v_h))
\coloneq
a_h(\utens{\sigma}_h,\utens{\tau}_h)
+ b_h(\utens{\tau}_h,u_h)
- b_h(\utens{\sigma}_h,v_h).
\]
In what follows, we equip the Cartesian product space $\utens{\Sigma}_h^\ell\times \Poly{\ell-1}(\Th)$ with the norm
\[
\norm[\btens{\Sigma}\times L,h]{(\utens{\tau}_h,v_h)}
\coloneq\left(
\norm[\btens{\Sigma},h]{\utens{\tau}_h}^2 + \norm[L^2(\Omega)]{v_h}^2
\right)^{\frac12}.
\]

  We next state the main results of the analysis, the proof of which is postponed to the coming sections.
  From this point on, to avoid the proliferation of generic constants, we write $a\lesssim b$ in place of $a\le Cb$ with $C$ possibly depending only on the domain $\Omega$ and on the mesh regularity parameter of \cite[Definition 1.9]{Di-Pietro.Droniou:20}, but independent of the meshsize and the physical parameters $D$, $\nu$.
  The following lemma contains existence and uniqueness of a discrete solution as well as an a priori estimate on it.
\begin{lemma}[Well-posedness of problem \eqref{eq:discrete}]\label{lem:well-posedness:discrete}
  The bilinear form $A_h$ satisfies the inf-sup condition: For all $(\utens{\sigma}_h,u_h)\in \utens{\Sigma}_h^\ell\times \Poly{\ell-1}(\Th)$, 
  \begin{equation}\label{eq:Ah.infsup}
    \gamma\norm[\btens{\Sigma}\times L,h]{(\utens{\sigma}_h,u_h)}
    \lesssim
    \sup_{(\utens{\tau}_h,v_h)\in \utens{\Sigma}_h^\ell\times \Poly{\ell-1}(\Th)\backslash\{(\btens{0},0)\}}    
    \frac{A_h((\utens{\sigma}_h,u_h),(\utens{\tau}_h,v_h))}{\norm[\btens{\Sigma}\times L,h]{(\utens{\tau}_h,v_h)}},
  \end{equation}
  where
  \begin{equation}\label{eq:def.gamma}
    \gamma\coloneq \left[ D^2\left(1+\tfrac{1}{D^2(1-\nu)^2}\right)^2+1\right]^{-\frac12}.
  \end{equation}
  As a consequence, there exists a unique solution $(\utens{\sigma}_h,u_h)\in\utens{\Sigma}_h^\ell\times \Poly{\ell-1}(\Th)$ to problem \eqref{eq:discrete} (or, equivalently, \eqref{eq:discrete:variational}), for which the following a priori bound holds:
  \begin{equation*}
    \norm[\btens{\Sigma}\times L,h]{(\utens{\sigma}_h,u_h)}
    \lesssim\gamma^{-1}\norm[L^2(\Omega)]{f}.
  \end{equation*}
\end{lemma}

\begin{proof}
  See Section \ref{sec:application:stability.analysis}.
\end{proof}

  The convergence rates of the DDR scheme for smooth enough solutions are estimated in the following theorem.

\begin{theorem}[Error estimate]\label{thm:error.estimate}
  Assume that the solution $(\btens{\sigma},u)$ to the continuous problem \eqref{eq:continuous.problem} is such that $\btens{\sigma}\in \btens{H}^2(\Omega;\Symm)\cap\btens{H}^{\ell+1}(\Th;\Symm)$ and $u\in C^1(\overline{\Omega})\cap H^{\ell+3}(\Th)$, and let $(\utens{\sigma}_h,u_h)\in\utens{\Sigma}_h^\ell\times \Poly{\ell-1}(\Th)$ be the solution to the discrete problem \eqref{eq:discrete} (or, equivalently, \eqref{eq:discrete:variational}). Then, recalling the definition \eqref{eq:def.gamma} of $\gamma$, it holds
  \begin{equation}\label{eq:error.estimate}
    \norm[\btens{\Sigma}\times L,h]{(\ISigma[h]{\ell}\btens{\sigma}-\utens{\sigma}_h,\pi^{\ell-1}_{\mathcal{P},h}u-u_h)}
    \lesssim\gamma^{-1} h^{\ell+1}\left(\tfrac{1}{D(1-\nu)}\seminorm[\btens{H}^{\ell+1}(\Th)]{\btens{\sigma}}+\seminorm[H^{\ell+3}(\Th)]{u}\right).
  \end{equation}
\end{theorem}

\begin{proof}
  See Section \ref{sec:application:convergence.analysis}.
\end{proof}

  \begin{remark}[Superconvergence of the deflection]
    The error estimate \eqref{eq:error.estimate} shows a superconvergence of $u_h$ to $\pi^{\ell-1}_{\mathcal{P},h}u$.
    This phenomenon is common in mixed formulations; see, e.g., \cite[Theorem 7]{Di-Pietro.Ern:17} concerning a Mixed High-Order method for the Darcy problem related to the rightmost portion of the DDR sequence of \cite{Di-Pietro.Droniou.ea:20}.
  \end{remark}


\subsection{Numerical examples}\label{sec:application:numerical.examples}

The DDR scheme \eqref{eq:discrete} was implemented in HArDCore (see \url{https://github.com/jdroniou/HArDCore}), an open source C++ library geared towards polytopal methods. For the resolution of the linear system, the moment unknowns in each cell (components in $\Holy{\ell-3}(T)\times \cHoly{\ell}(T)$ in $\utens{\Sigma}_h^\ell$) are eliminated using a standard static condensation process, and the resulting system is solved using the \texttt{Intel MKL PARDISO} library (see \url{https://software.intel.com/en-us/mkl}).
In order to numerically assess the convergence rates established in Theorem \ref{thm:error.estimate}, we have considered a manufactured solution of problem \eqref{eq:continuous.problem} with tensor $\mathbb{A}$ such that $\mathbb{A}\btens{\tau} = \btens{\tau}$ for all $\btens{\tau}\in\Symm$ and deflection field
\[
u(x_1,x_2) = \sin(\pi x_1)\sin(\pi x_2).
\]
The corresponding moment tensor field and orthogonal load are, respectively,
\[
\btens{\sigma}(x_1, x_2) = \pi^2\begin{pmatrix}
\sin(\pi x_1)\sin(\pi x_2) & -\cos(\pi x_1)\cos(\pi x_2) \\
-\cos(\pi x_1)\cos(\pi x_2) & \sin(\pi x_1)\sin(\pi x_2)
\end{pmatrix},\quad
f(x_1, x_2) = 4\pi^4\sin(\pi x_1)\sin(\pi x_2).
\]
To account for the fact that the normal derivative of $u$ does not vanish on $\partial\Omega$ (which corresponds to a non-homogeneous natural condition for the mixed formulation considered here), the right-hand side of \eqref{eq:discrete:flux} is replaced by $-\sum_{E\in\Eh^{\rm b}}\int_E \partial_{\normal} u~\PSigmaE{\ell}\utens{\tau}_E$, where $\Eh^{\rm b}$ is the set of boundary edges, $\partial_{\normal}u$ is the given boundary datum with $\normal:\partial\Omega\to\Real^2$ the outer normal to $\Omega$.
It can be easily checked that this modification does not alter the consistency results of Lemma \ref{lem:well-posedness:discrete}, hence Theorem \ref{thm:error.estimate} holds unchanged also in this case.

To showcase the ability of the DDR method to support both standard and general polygonal meshes (possibly including non-matching interfaces), we consider matching triangular, locally refined quadrangular, predominantly hexagonal, and Kershaw mesh sequences, an example of which is depicted in Figure \ref{fig:meshes}.

\begin{figure}\centering
  \begin{minipage}{0.24\textwidth}\centering
    \includegraphics[height=3.5cm]{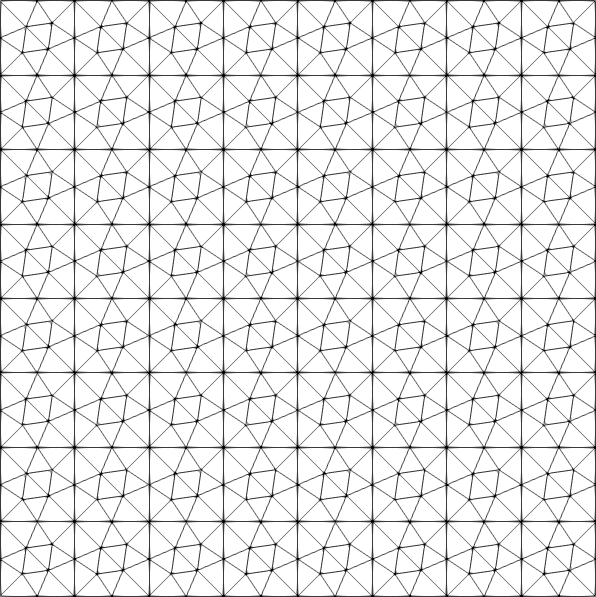}
    \subcaption{Matching triangular\label{fig:meshes:tri}}
  \end{minipage}
  \begin{minipage}{0.24\textwidth}\centering
    \includegraphics[height=3.5cm]{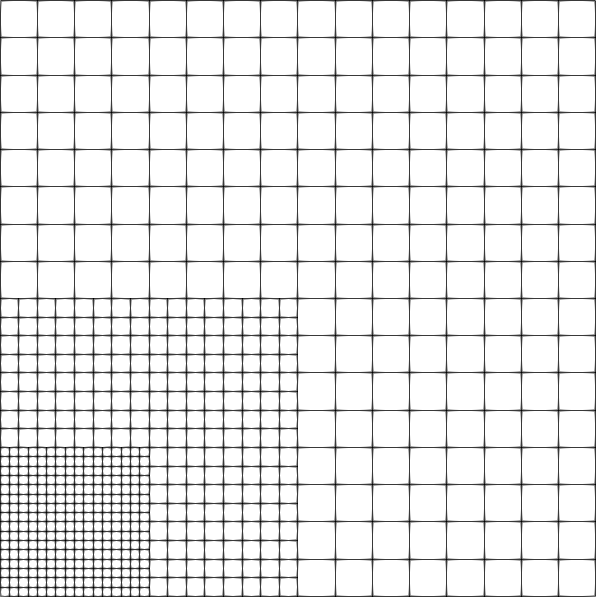}
    \subcaption{Locally refined\label{fig:meshes:locref}}
  \end{minipage}
  \begin{minipage}{0.24\textwidth}\centering
    \includegraphics[height=3.5cm]{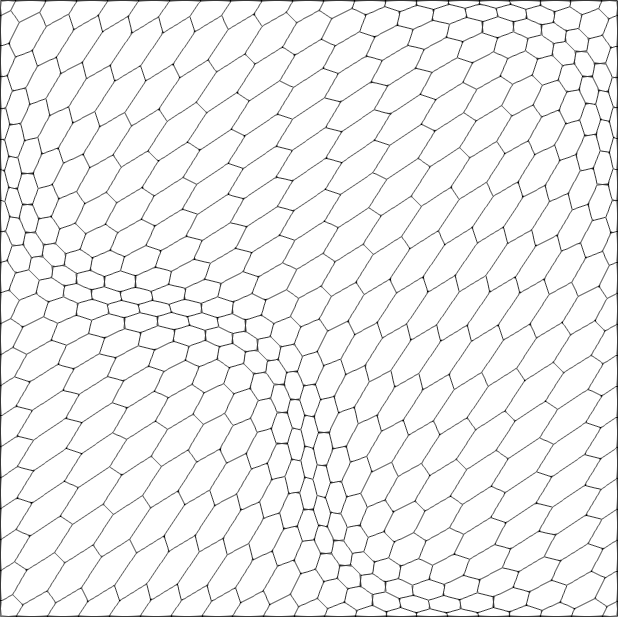}
    \subcaption{Hexagonal\label{fig:meshes:hexa}}
  \end{minipage}
  \begin{minipage}{0.24\textwidth}\centering
    \includegraphics[height=3.5cm]{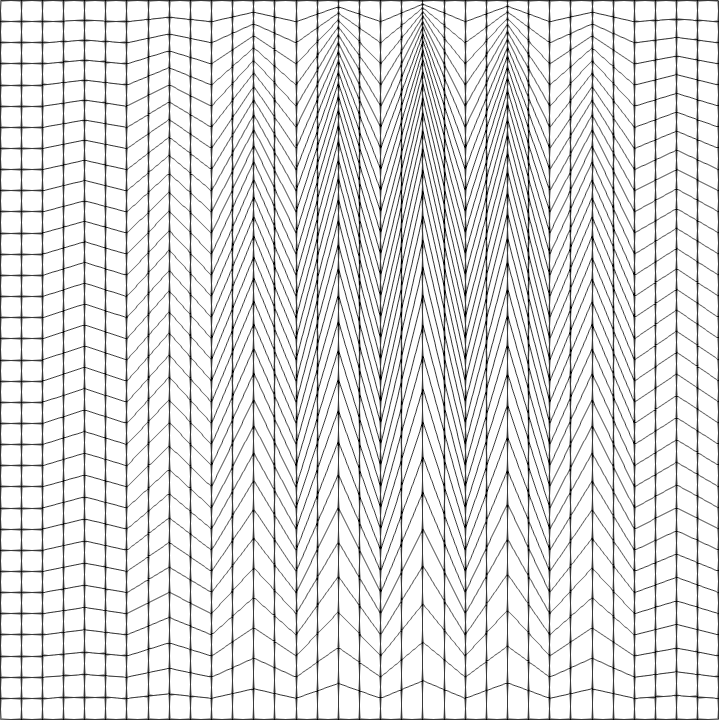}
    \subcaption{Kershaw\label{fig:meshes:kershaw}}
  \end{minipage}
  \caption{Meshes used in the numerical test of Section \ref{sec:application:numerical.examples}.
    The triangular, nonconforming, and Kershaw meshes are taken from the FVCA5 benchmark~\cite{Herbin.Hubert:08}.\label{fig:meshes}}
\end{figure}

The convergence of the energy error with respect to the meshsize $h$ for each of these mesh families and polynomial degrees $\ell\in\{2,3,4\}$ is displayed in Figure \ref{fig:trigonometric}.
The predicted orders of convergence are numerically verified for all the mesh sequences, as can be checked comparing each curve with the corresponding reference slope.

Saturation of the error can be noticed for $\ell=4$ on the finest triangular and Kershaw meshes.
This saturation is related to the accumulation of round-off errors and the different values at which it occurs is linked to the condition number of the system matrix.
  Notice that the Kershaw mesh family is particularly delicate, as the quality worsens at each refinement owing to the onset of more and more elongated elements.

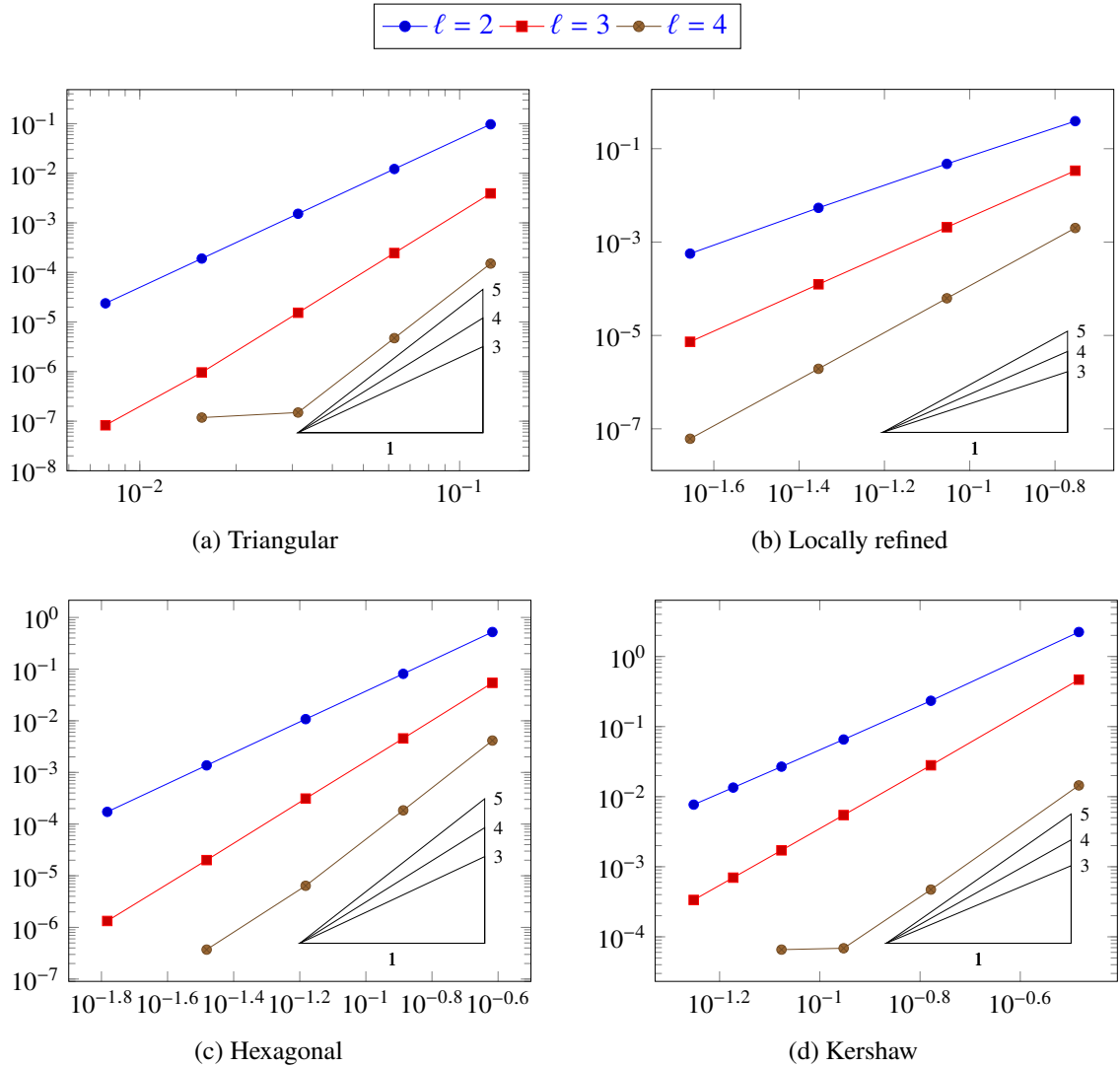
\begin{figure}\centering
  \ref{conv.trigo}
  \vspace{0.5cm}\\
  \begin{minipage}[b]{0.45\linewidth}\centering
    \begin{tikzpicture}[scale=0.90]
      \begin{loglogaxis}[ymin=10e-9]
        \addplot table[x=MeshSize,y=Error] {dat/tri_trigo_k3/data_rates.dat};
        \addplot table[x=MeshSize,y=Error] {dat/tri_trigo_k4/data_rates.dat};
        \addplot table[x=MeshSize,y=Error] {dat/tri_trigo_k5/data_rates.dat};
        \logLogSlopeTriangle{0.9}{0.4}{0.1}{3}{black};
        \logLogSlopeTriangle{0.9}{0.4}{0.1}{4}{black};        
        \logLogSlopeTriangle{0.9}{0.4}{0.1}{5}{black};        
      \end{loglogaxis}
    \end{tikzpicture}
    \subcaption{Triangular\label{fig:convergence:tri}}
  \end{minipage}
  \hspace{0.5cm}
  \begin{minipage}[b]{0.45\linewidth}\centering
    \begin{tikzpicture}[scale=0.90]
      \begin{loglogaxis}
        \addplot table[x=MeshSize,y=Error] {dat/locref_trigo_k3/data_rates.dat};
        \addplot table[x=MeshSize,y=Error] {dat/locref_trigo_k4/data_rates.dat};
        \addplot table[x=MeshSize,y=Error] {dat/locref_trigo_k5/data_rates.dat};
        \logLogSlopeTriangle{0.90}{0.4}{0.1}{3}{black};
        \logLogSlopeTriangle{0.90}{0.4}{0.1}{4}{black};
        \logLogSlopeTriangle{0.90}{0.4}{0.1}{5}{black};
      \end{loglogaxis}
    \end{tikzpicture}
    \subcaption{Locally refined\label{fig:convergence:locref}}
  \end{minipage}
  \vspace{0.5cm}\\
  \begin{minipage}[b]{0.45\linewidth}\centering
    \begin{tikzpicture}[scale=0.90]
      \begin{loglogaxis}
        \addplot table[x=MeshSize,y=Error] {dat/hexa_trigo_k3/data_rates.dat};
        \addplot table[x=MeshSize,y=Error] {dat/hexa_trigo_k4/data_rates.dat};
        \addplot table[x=MeshSize,y=Error] {dat/hexa_trigo_k5/data_rates.dat};
        \logLogSlopeTriangle{0.90}{0.4}{0.1}{3}{black};
        \logLogSlopeTriangle{0.90}{0.4}{0.1}{4}{black};
        \logLogSlopeTriangle{0.90}{0.4}{0.1}{5}{black};
      \end{loglogaxis}
    \end{tikzpicture}
    \subcaption{Hexagonal\label{fig:convergence:hexa}}
  \end{minipage}
  \hspace{0.5cm}
  \begin{minipage}[b]{0.45\linewidth}\centering
    \begin{tikzpicture}[scale=0.90]
      \begin{loglogaxis}[ legend columns=-1, legend to name=conv.trigo ]
        \addplot table[x=MeshSize,y=Error] {dat/kershaw_trigo_k3/data_rates.dat};
        \addplot table[x=MeshSize,y=Error] {dat/kershaw_trigo_k4/data_rates.dat};
        \addplot table[x=MeshSize,y=Error] {dat/kershaw_trigo_k5/data_rates.dat};
        \logLogSlopeTriangle{0.90}{0.4}{0.1}{3}{black};
        \logLogSlopeTriangle{0.90}{0.4}{0.1}{4}{black};
        \logLogSlopeTriangle{0.90}{0.4}{0.1}{5}{black};
        \legend{$\ell=2$,$\ell=3$,$\ell=4$};
      \end{loglogaxis}
    \end{tikzpicture}
    \subcaption{Kershaw\label{fig:convergence:kershaw}}
  \end{minipage}  
  \caption{Error $\norm[\btens{\Sigma}\times L,h]{(\ISigma[h]{\ell}\btens{\sigma}-\utens{\sigma}_h,\pi^{\ell-1}_{\mathcal{P},h}u-u_h)}$ vs.~meshsize $h$ for the test case of Section \ref{sec:application:numerical.examples} using the mesh families depicted in Figure \ref{fig:meshes}.
    The reference slopes refer to the orders of convergence predicted by Theorem \ref{thm:error.estimate} for each polynomial degree $\ell\in\{2,3,4\}$.\label{fig:trigonometric}}
\end{figure}


\subsection{Stability analysis}\label{sec:application:stability.analysis}

In this section we give a proof of Lemma \ref{lem:well-posedness:discrete} after discussing the required intermediate results.
Following analysis principles developed for the Discrete de Rham sequence \cite{Di-Pietro.Droniou:21*1}, we first introduce a component norm $\tnorm[\btens{\Sigma},h]{{\cdot}}$ on $\utens{\Sigma}_h^\ell$, establish boundedness with respect to this norm of discrete differential and potential operators, prove the equivalence of the component norm with the norm $\norm[\btens{\Sigma},h]{{\cdot}}$, and finally establish the boudedness of the interpolator on $\utens{\Sigma}_h^\ell$.

\subsubsection{Component norm and uniform equivalence with the operator norm}

The \emph{component norm} $\tnorm[\btens{\Sigma},h]{{\cdot}}:\utens{\Sigma}_h^\ell\to[0,\infty)$ is such that, for all $\utens{\tau}_h\in\utens{\Sigma}_h^\ell$,
\[
\tnorm[\btens{\Sigma},h]{\utens{\tau}_h}^2\coloneq \sum_{T\in\Th}\tnorm[\btens{\Sigma},T]{\utens{\tau}_T}^2
\]
with
\begin{align*}
  \tnorm[\btens{\Sigma},T]{\utens{\tau}_T}^2
  \coloneq{}&
  \norm[\btens{L}^2(T;\Real^{2\times 2})]{\btens{\tau}_{\cvec{H},T}}^2+\norm[\btens{L}^2(T;\Real^{2\times 2})]{\btens{\tau}_{\cvec{H},T}^\compl}^2
  \\
  &+\sum_{E\in\ET}\left(h_T\norm[L^2(E)]{\tau_E}^2+h_T^3\norm[L^2(E)]{D_{\btens{\tau},E}}^2\right)
  +\sum_{V\in\VT}h_T^2|\btens{\tau}_V|^2.
\end{align*}

\begin{proposition}[Boundedness of local operators]
  It holds, for all $T\in\Th$ and all $\utens{\tau}_T\in\utens{\Sigma}_T^\ell$,
  \begin{align}
    h_T^2\norm[L^2(T)]{\DD{\ell-1}\utens{\tau}_T}\lesssim{}& \tnorm[\btens{\Sigma},T]{\utens{\tau}_T},
    \label{eq:bound.DDT}\\
    \left(\sum_{E\in\ET}h_T\norm[L^2(E)]{\PSigmaE{\ell}\utens{\tau}_E}^2\right)^{\frac12}\lesssim{}& \tnorm[\btens{\Sigma},T]{\utens{\tau}_T},
    \label{eq:bound.PSigmaE}\\
    \norm[\btens{L}^2(T;\Real^{2\times 2})]{\PSigmaT{\ell}\utens{\tau}_T}\lesssim{}& \tnorm[\btens{\Sigma},T]{\utens{\tau}_T},
    \label{eq:bound.PSigmaT}
  \end{align}
  with $\utens{\tau}_E$ denoting the restriction of $\utens{\tau}_T$ to $E$.
\end{proposition}

\begin{proof}
  \underline{1. \emph{Proof of \eqref{eq:bound.DDT}.}}
  Making $q=\DD{\ell-1}\uvec{\tau}_T$ in the definition \eqref{eq:DDT} of $\DD{\ell}$ and using discrete inverse and trace inequalities on the edges and the vertices (see \cite[Lemmas 1.28 and 1.32]{Di-Pietro.Droniou:20}), we find
  \[
  \begin{aligned}
    \norm[L^2(T)]{\DD{\ell-1}\utens{\tau}_T}^2\lesssim{}& \norm[\btens{L}^2(T;\Real^{2\times 2})]{\btens{\tau}_{\cvec{H},T}}h_T^{-2}\norm[L^2(T)]{\DD{\ell-1}\utens{\tau}_T}+\sum_{E\in\ET}\sum_{V\in\VE} |\btens{\tau}_V|h_T^{-1}\norm[L^2(T)]{\DD{\ell-1}\utens{\tau}_T}\\
    &+\sum_{E\in\ET}\norm[L^2(E)]{\tau_E}h_T^{-\frac32}\norm[L^2(T)]{\DD{\ell-1}\utens{\tau}_T}
    +\sum_{E\in\ET}\norm[L^2(E)]{D_{\btens{\tau},E}}h_T^{-\frac12}\norm[L^2(T)]{\DD{\ell-1}\utens{\tau}_T}.
  \end{aligned}
  \]
  Simplyfying by $\norm[L^2(T)]{\DD{\ell-1}\utens{\tau}_T}$ and multiplying by $h_T^2$ leads to
  \[
  h_T^2\norm[L^2(T)]{\DD{\ell-1}\utens{\tau}_T}\lesssim \norm[\btens{L}^2(T;\Real^2)]{\btens{\tau}_{\cvec{H},T}}+\sum_{V\in\VT} h_T|\btens{\tau}_V|+\sum_{E\in\ET}h_T^{\frac12}\norm[L^2(E)]{\tau_E}+\sum_{E\in\ET}h_T^{\frac32}\norm[L^2(E)]{D_{\btens{\tau},E}}.
  \]
  The proof of \eqref{eq:bound.DDT} is completed using the Cauchy--Schwarz inequality on the sums together with $\card(\VT)\lesssim 1$ and $\card(\ET)\lesssim 1$.
  \medskip

  \underline{2. \emph{Proof of \eqref{eq:bound.PSigmaE}.}}
  The mapping $q\in\Poly{\ell}([0,1])\to (q(0),q(1),\pi_{[0,1]}^{\ell-2}q)\in\Real\times\Real\times\Poly{\ell-2}([0,1])$ is an isomorphism and thus, for all $q\in\Poly{\ell}([0,1])$, we have
  \[
  \norm[L^2(0,1)]{q}^2\lesssim \norm[L^2(0,1)]{\pi_{[0,1]}^{\ell-2}q}^2+|q(0)|^2+|q(1)|^2.
  \]
  For any $E\in\ET$, using the isomorphism $[0,1]\ni s\mapsto \bvec{x}_{V_1}+sh_E(\bvec{x}-\bvec{x}_{V_1})\in E$, where $V_1,V_2$ are the vertices of $E$, we infer from the definition \eqref{eq:P.Sigma.E} of $\PSigmaE{\ell}\utens{\tau}_E$ and $h_T\lesssim h_E$ that
  \[
  h_T\norm[L^2(E)]{\PSigmaE{\ell}\utens{\tau}_E}^2\lesssim h_T\norm[L^2(E)]{\tau_E}^2+h_T^2|\btens{\tau}_{V_1}|^2
  +h_T^2|\btens{\tau}_{V_2}|^2.
  \]
  Summing over $E\in\ET$ concludes the proof of \eqref{eq:bound.PSigmaE}.
  \medskip

  \underline{3. \emph{Proof of \eqref{eq:bound.PSigmaT}.}}
  Reasoning via a scaling/transport argument as in \cite[Section 2.5]{Di-Pietro.Droniou:21*1} (see in particular Lemma 2 therein), we easily see that \eqref{eq:PolySym=Holy.oplus.cHoly} is a topological decomposition; with $m=\ell$, this gives $(q,\btens{\upsilon})\in\Poly{\ell+2}(T)\times\cHoly{\ell}(T)$ such that $\HESS q+\btens{\upsilon}=\PSigmaT{\ell}\utens{\tau}_T$ and $\norm[\btens{L}^2(T;\Real^{2\times 2})]{\HESS q}+\norm[\btens{L}^2(T;\Real^{2\times 2})]{\btens{\upsilon}}\lesssim \norm[\btens{L}^2(T;\Real^{2\times 2})]{\PSigmaT{\ell}\utens{\tau}_T}$. The space
  \[
  \mathfrak{C}^{k+1}(T)\coloneq\SPAN\left\{\left(\frac{\bvec{x}-\bvec{x}_T}{h_T}\right)^{\bvec{\alpha}}\,:\,\bvec{\alpha}=(\alpha_1,\alpha_2)\in \Natural^2\,,\;2\le\alpha_1+\alpha_2\le k+1
  \right\}
  \]
  (where $\bvec{y}^{\bvec{\alpha}}=y_1^{\alpha_1}y_2^{\alpha_2}$ if $\bvec{y}=(y_1,y_2)\in\Real^2$) is a complement of $\Poly{1}(T)=\Ker\HESS$ in $\Poly{k+1}(T)$. Selecting $q\in \mathfrak{C}^{k+1}(T)$, we can again use a transport and scaling argument as in \cite[Lemma 9]{Di-Pietro.Droniou:21*1} to see that $\norm[L^2(T)]{q}\lesssim h_T^2\norm[\btens{L}^2(T;\Real^{2\times 2})]{\HESS q}$. Hence, we have $(q,\btens{\upsilon})\in\Poly{\ell+2}(T)\times\cHoly{\ell}(T)$ such that 
  \begin{equation}\label{eq:Holy.cHoly.topological}
    \HESS q+\btens{\upsilon}
    =\PSigmaT{\ell}\utens{\tau}_T\quad\mbox{ and }\quad
    h_T^{-2}\norm[L^2(T)]{q}+\norm[\btens{L}^2(T;\Real^{2\times 2})]{\btens{\upsilon}}\lesssim \norm[\btens{L}^2(T;\Real^{2\times 2})]{\PSigmaT{\ell}\utens{\tau}_T}.
  \end{equation}
  Plugging these $(q,\btens{\upsilon})$ into the definition \eqref{eq:P.Sigma.T} of $\PSigmaT{\ell}\utens{\tau}_T$ and using discrete inverse and trace inequalities on the edges and the vertices, we obtain
  \[
  \begin{aligned}
    \norm[\btens{L}^2(T;\Real^{2\times 2})]{\PSigmaT{\ell}\utens{\tau}_T}^2\lesssim{}&
    \norm[L^2(T)]{\DD{\ell-1}\utens{\tau}_T}\norm[L^2(T)]{q}+\sum_{E\in\ET}\sum_{V\in\VE}|\btens{\tau}_V|h_T^{-1}\norm[L^2(T)]{q}\\
    &+\sum_{E\in\ET}\left(\norm[L^2(E)]{\PSigmaE{\ell}\utens{\tau}_E}h_T^{-\frac32}\norm[L^2(T)]{q}+
    \norm[L^2(E)]{D_{\btens{\tau},E}}h_T^{-\frac12}\norm[L^2(T)]{q}\right)\\
    &+\norm[\btens{L}^2(T;\Real^{2\times 2})]{\btens{\tau}_{\cvec{H},T}^\compl}\norm[\btens{L}^2(T;\Real^{2\times 2})]{\btens{\upsilon}}.
  \end{aligned}
  \]
  Applying \eqref{eq:Holy.cHoly.topological}, simplifying by $\norm[\btens{L}^2(T;\Real^{2\times 2})]{\PSigmaT{\ell}\utens{\tau}_T}$, and using Cauchy--Schwarz inequalities on the sums, we infer
  \[
  \begin{aligned}
    \norm[\btens{L}^2(T;\Real^{2\times 2})]{\PSigmaT{\ell}\utens{\tau}_T}\lesssim{}&
    h_T^2\norm[L^2(T)]{\DD{\ell-1}\utens{\tau}_T}+\left(\sum_{V\in\VT}h_T^2|\btens{\tau}_V|^2\right)^{\frac12}\\
    &+\left(\sum_{E\in\ET}h_T\norm[L^2(E)]{\PSigmaE{\ell}\utens{\tau}_E}^2\right)^{\frac12}+
    \left(\sum_{E\in\ET}h_T^3\norm[L^2(E)]{D_{\btens{\tau},E}}^2\right)^{\frac12}+\norm[\btens{L}^2(T;\Real^{2\times 2})]{\btens{\tau}_{\cvec{H},T}^\compl}.
  \end{aligned}
  \]
  The estimate \eqref{eq:bound.PSigmaT} then follows from \eqref{eq:bound.DDT} and \eqref{eq:bound.PSigmaE}.
\end{proof}

\begin{lemma}[Uniform equivalence between the component and operator norms]\label{lem:norm.equivalence}
  It holds, for $\bullet=h$ or $T\in\Th$,
  \begin{equation}\label{eq:norm.equivalence}
    \tnorm[\btens{\Sigma},\bullet]{\utens{\tau}_\bullet}\simeq \norm[\btens{\Sigma},\bullet]{\utens{\tau}_\bullet}\qquad\forall \utens{\tau}_\bullet\in\utens{\Sigma}_\bullet^\ell,
  \end{equation}
  where $a\simeq b$ means $a\lesssim b$ and $b\lesssim a$.
\end{lemma}

\begin{proof}
  We obviously only have to prove the case $\bullet=T$ for a generic $T$, the case $\bullet=h$ resulting by squaring and summing over $T\in\Th$.
  We have
  \begin{equation}\label{eq:norm.tau.T}
    \begin{aligned}
      \norm[\btens{\Sigma},T]{\utens{\tau}_T}^2={}&\norm[\btens{L}^2(T;\Real^{2\times 2})]{\PSigmaT{\ell}\utens{\tau}_T}^2+
      \sum_{E\in\ET}h_T\norm[L^2(E)]{\PSigmaT{\ell}\utens{\tau}_T\normal_E\cdot\normal_E - \PSigmaE{\ell}\utens{\tau}_E}^2\\
      &+ \sum_{E\in\ET}h_T^3\norm[L^2(E)]{\partial_{\tangent_E}(\PSigmaT{\ell}\utens{\tau}_T\normal_E\cdot\tangent_E) + \VDIV\PSigmaT{\ell}\utens{\tau}_T\cdot\normal_E - D_{\btens{\tau},E}}^2\\
      &+ \sum_{V\in\VT} h_T^2|\PSigmaT{\ell}\utens{\tau}_T(\bvec{x}_V) - \btens{\tau}_V|^2.
    \end{aligned}
  \end{equation}
  Using triangle, discrete trace, and discrete inverse inequalities, we infer
  \[
    \norm[\btens{\Sigma},T]{\utens{\tau}_T}^2
    \lesssim
    \norm[\btens{L}^2(T;\Real^{2\times 2})]{\PSigmaT{\ell}\utens{\tau}_T}^2
    + \sum_{E\in\ET}h_T\norm[L^2(E)]{\PSigmaE{\ell}\utens{\tau}_E}^2
    + \sum_{E\in\ET}h_T^3\norm[L^2(E)]{D_{\btens{\tau},E}}^2
    + \sum_{V\in\VT} h_T^2|\btens{\tau}_V|^2.
  \]
  The estimates \eqref{eq:bound.PSigmaT} (for the first term in the right-hand side) and \eqref{eq:bound.PSigmaE} (for the second term) conclude the proof that $\norm[\btens{\Sigma},T]{\utens{\tau}_T}\lesssim\tnorm[\btens{\Sigma},T]{\utens{\tau}_T}$.
  \medskip

  Let us turn to the converse inequality. Making $\bvec{\upsilon}=\bvec{0}$ and taking a generic $q\in\Poly{\ell-1}(T)$ in the definition \eqref{eq:P.Sigma.T} of $\PSigmaT{\ell}\utens{\tau}_T$, the definitions \eqref{eq:P.Sigma.E} of $\PSigmaE{\ell}$ and \eqref{eq:DDT} of $\DD{\ell-1}$ show that
  \[
  \int_T\PSigmaT{\ell}\utens{\tau}_T:\HESS q = \int_T \btens{\tau}_{\cvec{H},T}:\HESS q,
  \]
  which yields $\btens{\pi}_{\cvec{H},T}^{\ell-3}(\PSigmaT{\ell}\utens{\tau}_T)= \btens{\tau}_{\cvec{H},T}$.
  Making $q=0$ in the definition \eqref{eq:P.Sigma.T} of $\PSigmaT{\ell}\utens{\tau}_T$, we also see that $\btens{\pi}_{\cvec{H},T}^{\compl,\ell}(\PSigmaT{\ell}\utens{\tau}_T) = \btens{\tau}_{\cvec{H},T}^\compl$.
  The $L^2$-boundedness of the orthogonal projectors thus yield
  \[
  \norm[\btens{L}^2(T;\Real^{2\times 2})]{\btens{\tau}_{\cvec{H},T}}^2+\norm[\btens{L}^2(T;\Real^{2\times 2})]{\btens{\tau}_{\cvec{H},T}^\compl}^2
  \le 2\norm[\btens{L}^2(T;\Real^{2\times 2})]{\PSigmaT{\ell}\utens{\tau}}^2.
  \]
  Using the definition \eqref{eq:P.Sigma.E} of $\PSigmaE{\ell}\utens{\tau}_E$ and the $L^2$-boundedness of $\pi^{\ell-2}_{\mathcal{P},E}$, we therefore infer
  \[
  \tnorm[\btens{\Sigma},T]{\utens{\tau}_T}^2\le
  2\norm[\btens{L}^2(T;\Real^{2\times 2})]{\PSigmaT{\ell}\utens{\tau}_T}^2+\sum_{E\in\ET}\left(h_T\norm[L^2(E)]{\PSigmaE{\ell}\utens{\tau}_E}^2+h_T^3\norm[L^2(E)]{D_{\btens{\tau},E}}^2\right)
  +\sum_{V\in\VT}h_T^2|\btens{\tau}_V|^2.
  \]
  Introducing $\PSigmaT{\ell}\utens{\tau}_T\normal_E\cdot\normal_E$, $\partial_{\tangent_E}(\PSigmaT{\ell}\utens{\tau}_T\normal_E\cdot\tangent_E) + \VDIV\PSigmaT{\ell}\utens{\tau}_T\cdot\normal_E$, and $\PSigmaT{\ell}\utens{\tau}_T(\bvec{x}_V)$ respectively in the edge and vertex terms, using triangle inequalities, and recalling \eqref{eq:norm.tau.T}, we infer
  \begin{multline*}
    \tnorm[\btens{\Sigma},T]{\utens{\tau}_T}^2\lesssim
    \norm[\btens{\Sigma},T]{\utens{\tau}_T}^2
    +\sum_{E\in\ET}h_T\norm[L^2(E)]{\PSigmaT{\ell}\utens{\tau}_T\normal_E\cdot\normal_E}^2\\
    +\sum_{E\in\ET}h_T^3\norm[L^2(E)]{\partial_{\tangent_E}(\PSigmaT{\ell}\utens{\tau}_T\normal_E\cdot\tangent_E) + \VDIV\PSigmaT{\ell}\utens{\tau}_T\cdot\normal_E}^2
    +\sum_{V\in\VT}h_T^2|\PSigmaT{\ell}\utens{\tau}_T(\bvec{x}_V)|^2.
  \end{multline*}
  The proof is completed applying discrete trace and inverse inequalities to bound the sums in the right-hand side, up to a multiplicative constant, by $\norm[\btens{L}^2(T;\Real^{2\times 2})]{\PSigmaT{\ell}\utens{\tau}_T}\le\norm[\btens{\Sigma},T]{\utens{\tau}_T}$.
\end{proof}

\subsubsection{Boundedness of the interpolator and inf-sup condition on $b_h$}

\begin{proposition}[Boundedness of {$\ISigma{\ell}$}]\label{eq:boundedness:ISigma}
  For all $T\in\Th$ and all $\btens{\tau}\in \btens{H}^2(T;\Symm)$, it holds
   \begin{equation}\label{eq:ISigmaT:boundedness}
      \norm[\btens{\Sigma},T]{\ISigma{\ell}\btens{\tau}}\lesssim 
      \norm[\btens{L}^2(T;\Real^{2\times 2})]{\btens{\tau}}+
      h_T\seminorm[\btens{H}^1(T;\Real^{2\times 2})]{\btens{\tau}}+
      h_T^2\seminorm[\btens{H}^2(T;\Real^{2\times 2})]{\btens{\tau}}.
   \end{equation}
  As a consequence, for all $\btens{\tau}\in \btens{H}^2(\Omega;\Symm)$,
  \begin{equation*}
    \norm[\btens{\Sigma},h]{\ISigma[h]{\ell}\btens{\tau}}
    \lesssim\norm[\btens{H}^2(\Omega;\Real^{2\times 2})]{\btens{\tau}}.
  \end{equation*}
\end{proposition}

\begin{proof}
By the norm equivalence of Lemma \ref{lem:norm.equivalence}, $\norm[\btens{\Sigma},T]{\ISigma{\ell}\btens{\tau}}\lesssim\tnorm[\btens{\Sigma},T]{\ISigma{\ell}\btens{\tau}}$.
Moreover, using the $L^2$-boundedness of orthogonal projectors and the definition \eqref{eq:ISigmaT} of $\ISigma{\ell}\btens{\tau}$, we have
 \[
 \begin{aligned}
 \tnorm[\btens{\Sigma},T]{\ISigma{\ell}\btens{\tau}}^2
    \lesssim{}&
    \norm[\btens{L}^2(T;\Real^{2\times 2})]{\btens{\tau}}^2+
    \sum_{E\in\ET}\left(h_T\norm[\btens{L}^2(E;\Real^{2\times 2})]{\btens{\tau}}^2
    + h_T^3\norm[\btens{L}^2(E;\Real^{2\times 2})]{\partial_{\tangent_E}\btens{\tau}_{|E}}^2
    + h_T^3\norm[\bvec{L}^2(E;\Real^2)]{\VDIV\btens{\tau}}^2\right)\\
    &+\sum_{V\in\VT}h_T^2|\btens{\tau}(\bvec{x}_V)|^2.
  \end{aligned}
 \]
 Continuous trace inequalities (see, e.g., \cite[Lemma 1.31]{Di-Pietro.Droniou:20}) applied to $\btens{\tau}$ and its derivatives lead to
 \begin{equation}\label{eq:bound.ISigma.1}
 \tnorm[\btens{\Sigma},T]{\ISigma{\ell}\btens{\tau}}^2
    \lesssim
    \norm[\btens{L}^2(T;\Real^{2\times 2})]{\btens{\tau}}^2+h_T^2\seminorm[\btens{H}^1(T;\Real^{2\times 2})]{\btens{\tau}}^2+h_T^4\seminorm[\btens{H}^2(T;\Real^{2\times 2})]{\btens{\tau}}^2
    +\sum_{V\in\VT}h_T^2|\btens{\tau}(\bvec{x}_V)|^2.
 \end{equation}
 By \cite[Eq.~(5.110)]{Di-Pietro.Droniou:20} and mesh regularity (which gives $|T|^{-\frac12}\lesssim h_T^{-1}$), it holds, for all $V\in\VT$,
  \begin{equation}\label{eq:bound.tau.xV}
    |\btens{\tau}(\bvec{x}_V)|\lesssim h_T^{-1} \norm[\btens{L}^2(T;\Real^{2\times 2})]{\btens{\tau}}
    +\seminorm[\btens{H}^1(T;\Real^{2\times 2})]{\btens{\tau}}+h_T \seminorm[\btens{H}^2(T;\Real^{2\times 2})]{\btens{\tau}}.
  \end{equation}
 Plugging \eqref{eq:bound.tau.xV} into \eqref{eq:bound.ISigma.1} completes the proof of \eqref{eq:ISigmaT:boundedness}.
\end{proof}

\begin{lemma}[Discrete inf-sup condition for $b_h$]
  It holds, for all $v_h\in \Poly{\ell-1}(\Th)$,
  \begin{equation}\label{eq:inf-sup}
    \norm[L^2(\Omega)]{v_h}
    \lesssim
    \sup_{\utens{\tau}_h\in\utens{\Sigma}_h^\ell\setminus\{\utens{0}\}} \frac{b_h(\utens{\tau}_h,v_h)}{\norm[\btens{\Sigma},h]{\utens{\tau}_h}}.    
  \end{equation}
\end{lemma}
\begin{proof}
By \cite{Chen.Huang:18}, $\DIV\VDIV:\Hdivdiv{\Omega}{\Symm}\to L^2(\Omega)$ is surjective, i.e., there exists $\btens{\tau}_v\in \btens{H}^2(T;\Symm)$ such that
  $\DIV\VDIV\btens{\tau}_v = v_h$ and $\norm[\btens{H}^2(\Omega;\Real^{2\times 2})]{\btens{\tau}_v}\lesssim\norm[L^2(\Omega)]{v_h}$.
  Combining this fact with the commutation property \eqref{eq:DDT.commutation} of $\DD{\ell-1}$ and the boundedness of the interpolator proved in Proposition \ref{eq:boundedness:ISigma}, one can use the classical Fortin argument (see, e.g., \cite[Section 5.4.3]{Boffi.Brezzi.ea:13}) to infer \eqref{eq:inf-sup}.
\end{proof}

\subsubsection{Proof of Lemma \ref{lem:well-posedness:discrete}}

Owing to \eqref{eq:A.coer.bound}, the bilinear form $a_h$ defined by \eqref{eq:ah.bh} is clearly coercive and continuous with respect to the $\norm[\btens{\Sigma},h]{{\cdot}}$-norm, with coercivity constant $\frac{1}{D(1+\nu)}$ and continuity constant $\frac{2}{D(1-\nu)}$.
By \eqref{eq:inf-sup}, the bilinear form $b_h$ is inf-sup stable with a constant that does not depend on $D$, $\nu$ or the meshsize.
The conclusion follows applying \cite[Lemma A.11 and Proposition A.4]{Di-Pietro.Droniou:20}.


\subsection{Convergence analysis}\label{sec:application:convergence.analysis}

  This section contains technical results concerning the consistency of the potential reconstruction in $\utens{\Sigma}_h^\ell$ and of the stabilisation bilinear forms $\{s_{\btens{\Sigma},T}\}_{T\in\Th}$, from which a bound on the consistency error is inferred.
  Theorem \ref{thm:error.estimate} directly follows combining the inf-sup stability \eqref{eq:Ah.infsup} along with the latter bound (see \eqref{eq:consistency.bound} below) and applying the Third Strang Lemma of \cite{Di-Pietro.Droniou:18}.

\begin{proposition}[Consistency of the potential reconstruction and stabilization form]
  For all $T\in\Th$ and all $\btens{\tau}\in \btens{H}^{\ell+1}(T;\Symm)$, it holds
  \begin{align}\label{eq:consistency.PT}
    \norm[\btens{L}^2(T;\Real^{2\times 2})]{\PSigmaT{\ell}(\ISigma{\ell}\btens{\tau})-\btens{\tau}}\lesssim
         {}&h_T^{\ell+1}\seminorm[\btens{H}^{\ell+1}(T)]{\btens{\tau}},\\
         \label{eq:consistency.sT}
         s_{\btens{\Sigma},T}(\ISigma{\ell}\bvec{\tau},\utens{\upsilon}_T)\lesssim{}&
         h_T^{\ell+1}\seminorm[\btens{H}^{\ell+1}(T)]{\btens{\tau}}\norm[\btens{\Sigma},T]{\utens{\upsilon}_T}
         \qquad\forall \utens{\upsilon}_T\in\utens{\Sigma}_T^\ell.
  \end{align}
\end{proposition}

\begin{proof}
  \underline{1. \emph{Proof of \eqref{eq:consistency.PT}.}}
  Let $\hat{\btens{\tau}}_T\coloneq\btens{\pi}^\ell_{\ctens{P},T}\btens{\tau}$ be the $L^2$-orthogonal projection of $\btens{\tau}$ on $\tPoly{2}(T;\Symm)$ (which boils down to the component-wise $L^2$-projections). The polynomial consistency \eqref{eq:PT.poly.consistency} of $\PSigmaT{\ell}$ yields $\PSigmaT{\ell}(\ISigma{\ell}\hat{\btens{\tau}}_T) = \hat{\btens{\tau}}_T$, and so
  \[
  \PSigmaT{\ell}(\ISigma{\ell}\btens{\tau})-\btens{\tau}=
  \PSigmaT{\ell}(\ISigma{\ell}(\btens{\tau}-\hat{\btens{\tau}}_T))-(\btens{\tau}-\hat{\btens{\tau}}_T).
  \]
  Applying a triangle inequality and the estimate \eqref{eq:bound.PSigmaT} on $\PSigmaT{\ell}$, the norm equivalence \eqref{eq:norm.equivalence}, and the boundedness \eqref{eq:ISigmaT:boundedness} of $\ISigma{\ell}$, we infer
  \begin{align*}
    \norm[\btens{L}^2(T;\Real^{2\times 2})]{\PSigmaT{\ell}(\ISigma{\ell}\btens{\tau})-\btens{\tau}}
    \lesssim{}&
    \norm[\btens{\Sigma},T]{\ISigma{\ell}(\btens{\tau}-\hat{\btens{\tau}}_T)}+
    \norm[\btens{L}^2(T;\Real^{2\times 2})]{\btens{\tau}-\hat{\btens{\tau}}_T}\\
    \lesssim{}&
    \norm[\btens{L}^2(T;\Real^{2\times 2})]{\btens{\tau}-\hat{\btens{\tau}}_T}+
    h_T\seminorm[\btens{H}^1(T;\Real^{2\times 2})]{\btens{\tau}-\hat{\btens{\tau}}_T}+
    h_T^2\seminorm[\btens{H}^2(T;\Real^{2\times 2})]{\btens{\tau}-\hat{\btens{\tau}}_T}.
  \end{align*}
  The estimate \eqref{eq:consistency.PT} follows from the approximation properties of $\btens{\pi}^\ell_{\mathcal{P},T}$, see \cite[Theorem 1.45]{Di-Pietro.Droniou:20} and also \cite[Lemmas 3.4 and 3.6]{Di-Pietro.Droniou:17}.
  \medskip\\
  \underline{2. \emph{Proof of \eqref{eq:consistency.sT}.}}
  To prove \eqref{eq:consistency.sT}, we first notice that the polynomial consistencies \eqref{eq:PE.poly.consistency} and \eqref{eq:PT.poly.consistency} yield
  \[
  s_{\btens{\Sigma},T}(\ISigma{\ell}\hat{\btens{\tau}}_T,\utens{\upsilon}_T)=0\qquad
  \forall(\hat{\btens{\tau}}_T,\utens{\upsilon}_T)\in\tPoly{\ell}(T;\Symm)\times\utens{\Sigma}_T^\ell.
  \]
  Letting, as above, $\hat{\btens{\tau}}_T = \btens{\pi}^\ell_{\ctens{P},T}\btens{\tau}$, we infer that
  \begin{align*}
    s_{\btens{\Sigma},T}(\ISigma{\ell}\btens{\tau},\ISigma{\ell}\btens{\tau})^{\frac12}
    = s_{\btens{\Sigma},T}(\ISigma{\ell}(\btens{\tau}-\hat{\btens{\tau}}_T),\ISigma{\ell}(\btens{\tau}-\hat{\btens{\tau}}_T))^{\frac12}
    \le\norm[\btens{\Sigma},T]{\ISigma{\ell}(\btens{\tau}-\hat{\btens{\tau}}_T)},
  \end{align*}
  where conclusion follows from the definition of $\norm[\btens{\Sigma},T]{{\cdot}}$.
  The approximation properties of $\btens{\pi}^\ell_{\ctens{P},T}$ then yield, as in Point 1., $s_{\btens{\Sigma},T}(\ISigma{\ell}\btens{\tau},\ISigma{\ell}\btens{\tau})^{\frac12}\lesssim h_T^{\ell+1}\seminorm[\btens{H}^{\ell+1}(T;\Real^{2\times 2})]{\btens{\tau}}$, and \eqref{eq:consistency.sT} follows using a Cauchy--Schwarz inequality on the positive semi-definite bilinear form $s_{\btens{\Sigma},T}$.
\end{proof}

\begin{lemma}[Consistency error bound]\label{lem:consistency.bound}
  Assume that $\btens{\sigma}\in \btens{H}^2(\Omega;\Symm)\cap\btens{H}^{\ell+1}(\Th;\Symm)$ and $u\in C^1(\overline{\Omega})\cap H^{\ell+3}(\Th)$ solve \eqref{eq:strong.problem}.
  Let the consistency error $\mathcal E_h((\btens{\sigma},u);\cdot):\utens{\Sigma}_h^\ell\times \Poly{\ell-1}(\Th)\to\Real$ be such that, for all $(\utens{\tau}_h,v_h)\in \utens{\Sigma}_h^\ell\times \Poly{\ell-1}(\Th)$,
  \begin{align*}
    \mathcal E_h((\btens{\sigma},u);(\utens{\tau}_h,v_h))\coloneq{}&\int_\Omega f v_h-a_h(\ISigma[h]{\ell}\btens{\sigma},\utens{\tau}_h)-b_h(\utens{\tau}_h,\pi^{\ell-1}_{\mathcal{P},h}u) + b_h(\ISigma[h]{\ell}\btens{\sigma},v_h).
  \end{align*}
  Then, it holds
  \begin{equation}\label{eq:consistency.bound}
    |\mathcal E_h((\btens{\sigma},u);(\utens{\tau}_h,v_h))|\lesssim h^{\ell+1}\norm[\btens{\Sigma}\times L,h]{(\utens{\tau}_h,v_h)}\left(
    \tfrac{1}{D(1-\nu)}\seminorm[\btens{H}^{\ell+1}(\Th)]{\btens{\sigma}}
    + \seminorm[H^{\ell+3}(\Th)]{u}
    \right).
  \end{equation}
\end{lemma}

\begin{proof}
  By definition of $b_h$ and the commutation property \eqref{eq:DDT.commutation},
  \[
  b_h(\ISigma[h]{\ell}\btens{\sigma},v_h)=
  \sum_{T\in\Th}\int_T \DD{\ell-1}(\ISigma{\ell}\btens{\sigma})\, v_T
  =\sum_{T\in\Th}\int_T \cancel{\pi^{\ell-1}_{\mathcal{P},T}}(\DIV\VDIV\btens{\sigma})\, v_T
  =\int_\Omega\DIV\VDIV\btens{\sigma}\,v_h,
  \]
  where the cancellation of the projector is justified since $v_T\in\Poly{\ell-1}(T)$.
  Using $-\DIV\VDIV\btens{\sigma}=f$, we see that the first and last terms in the definition of the consistency error cancel out, and thus that
  \begin{align*}
    \mathcal E_h((\btens{\sigma},u);(\utens{\tau}_h,v_h))={}&
    - \sum_{T\in\Th}\left(
    \int_T\mathbb{A}^{-1}\PSigmaT{\ell}(\ISigma{\ell}\btens{\sigma}):\PSigmaT{\ell}\utens{\tau}_T
    + \frac{1}{D(1+\nu)} s_{\btens{\Sigma},T}(\ISigma{\ell}\btens{\sigma},\utens{\tau}_T)
    \right)\\
  &-\sum_{T\in\Th}\int_T \DD{\ell-1}\utens{\tau}_T\,\pi^{\ell-1}_{\mathcal{P},T}u.
  \end{align*}
  We then add and subtract $\mathbb{A}^{-1}\btens{\sigma}=-\HESS u$ to $\PSigmaT{\ell}(\ISigma{\ell}\btens{\sigma})$ and get
  \begin{align}
      \mathcal E_h((\btens{\sigma},u);(\utens{\tau}_h,v_h))={}&
      \sum_{T\in\Th}\left(
      \int_T\mathbb{A}^{-1}(\btens{\sigma}-\PSigmaT{\ell}(\ISigma{\ell}\btens{\sigma})):\PSigmaT{\ell}\utens{\tau}_T
      + \frac{1}{D(1+\nu)} s_{\btens{\Sigma},T}(\ISigma{\ell}\btens{\sigma},\utens{\tau}_T)
      \right)
      \nonumber\\
      &+\left(
      \sum_{T\in\Th}\int_T \HESS u:\PSigmaT{\ell}\utens{\tau}_T-\sum_{T\in\Th}\int_T \DD{\ell-1}\utens{\tau}_T\,\cancel{\pi^{\ell-1}_{\mathcal{P},T}}u\right)
      \eqcolon\term_1+\term_2,
  \label{eq:decompose.calE}
  \end{align}
  the cancellation of the projector being justified by $\DD{\ell-1}\utens{\tau}_T\in\Poly{\ell-1}(T)$.
  The consistency properties \eqref{eq:consistency.PT} and \eqref{eq:consistency.sT} of the local potential reconstruction and stabilization forms together with Cauchy--Schwarz inequalities and \eqref{eq:A.coer.bound} yield
  \begin{equation}\label{eq:term12}
    |\term_1|\lesssim \tfrac{1}{D(1-\nu)} h^{\ell+1}\seminorm[\btens{H}^{\ell+1}(\Th)]{\btens{\sigma}}\norm[\btens{\Sigma},h]{\utens{\tau}_h}.
  \end{equation}
  To estimate $\term_2$, we set $\hat{u}_T=\pi^{\ell+2}_{\mathcal{P},T}u$ and notice that the definition \eqref{eq:P.Sigma.T} of $\PSigmaT{\ell}$ yields
  \[
  \begin{aligned}
    0={}&-\int_T\PSigmaT{\ell}\utens{\tau}_T:\HESS \hat{u}_T
    +\int_T\DD{\ell-1}\utens{\tau}_T\,\hat{u}_T
    + \sum_{E\in\ET}\omega_{TE}\sum_{V\in\VE}\omega_{EV}(\btens{\tau}_V\normal_E\cdot\tangent_E)\,\hat{u}_T(\bvec{x}_V)
    \\
    &
    + \sum_{E\in\ET}\omega_{TE}\left(
    \int_E \PSigmaE{\ell}\utens{\tau}_E\,\partial_{\normal_E} \hat{u}_T
    - \int_E D_{\btens{\tau},E}\,\hat{u}_T
    \right).
  \end{aligned}
  \]
  Adding this quantity to $\term_2$ leads to
  \begin{equation}\label{eq:term2.first}
    \begin{aligned}
      \term_2={}&\left[\sum_{T\in\Th}\int_T \HESS (u-\hat{u}_T):\PSigmaT{\ell}\utens{\tau}_T-\sum_{T\in\Th}\int_T \DD{\ell-1}\utens{\tau}_T\,(u-\hat{u}_T)\right]\\
           {}&+ \sum_{T\in\Th}\sum_{E\in\ET}\omega_{TE}\sum_{V\in\VE}\omega_{EV}(\btens{\tau}_V\normal_E\cdot\tangent_E)\,(\hat{u}_T(\bvec{x}_V)-u(\bvec{x}_V))\\
           & + \sum_{T\in\Th}\sum_{E\in\ET}\omega_{TE}\left(
           \int_E \PSigmaE{\ell}\utens{\tau}_E\,\partial_{\normal_E} (\hat{u}_T-u)
           - \int_E D_{\btens{\tau},E}\,(\hat{u}_T-u)
           \right)\\
           \eqcolon{}&\term_{21}+\term_{22}+\term_{23},
    \end{aligned}
  \end{equation}
  where the introduction of $u(\bvec{x}_V)$ in $\term_{22}$ is justified by the fact that it vanishes for boundary vertices and that, if $T_1,T_2$ are the two elements on each side of an internal edge $E$,
  \[
  \sum_{T\in\Th}\sum_{E\in\ET}\omega_{TE}\underbrace{\sum_{V\in\VE}\omega_{EV}(\btens{\tau}_V\normal_E\cdot\tangent_E)\,u(\bvec{x}_V)}_{=:a_{EV}}
  =\sum_{E\in\Eh}\underbrace{(\omega_{T_1E}+\omega_{T_2E})}_{=0}a_{EV}=0,
  \]
  while the introduction of $u$ in $\term_{23}$ follows from the single-valuedness of this function and its derivatives on the edges along with $u = \partial_{\normal} u = 0$ on $\partial\Omega$, and a similar argument.
  The approximation properties of $\pi^{\ell+2}_{\mathcal{P},T}$ and the bound \eqref{eq:bound.DDT} on $\DD{\ell-1}$ give
  \begin{equation}\label{eq:est.term21}
    \begin{aligned}
      |\term_{21}|\lesssim{}& \sum_{T\in\Th}\left(h_T^{\ell+1}\seminorm[H^{\ell+3}(T)]{u}\norm[\btens{L}^2(T;\Real^{2\times 2})]{\PSigmaT{\ell}\utens{\tau}_T}+h_T^{\ell+3}\seminorm[H^{\ell+3}(T)]{u}\norm[L^2(T)]{\DD{\ell-1}\utens{\tau}_T}\right)\\
      \lesssim{}& h^{\ell+1}\seminorm[H^{\ell+3}(\Th)]{u}\norm[\btens{\Sigma},h]{\utens{\tau}_h}.
    \end{aligned}
  \end{equation}
  To estimate $\term_{22}$, we use \eqref{eq:bound.tau.xV} and the approximation properties of $\pi^{\ell+2}_{\mathcal{P},T}$:
  \begin{align}
    |\term_{22}|\lesssim\sum_{T\in\Th}\sum_{V\in\VT}|\btens{\tau}_V|h_T^{\ell+2}\seminorm[H^{\ell+3}(T)]{u}
    \le{}& \left(\sum_{T\in\Th}\sum_{V\in\VT}h_T^2|\btens{\tau}_V|^2\right)^{\frac12}
      \left(\sum_{T\in\Th}h_T^{2(\ell+1)}\seminorm[H^{\ell+3}(T)]{u}^2\right)^{\frac12}\nonumber\\
      \lesssim{}&\norm[\btens{\Sigma},h]{\utens{\tau}_h}h^{\ell+1}\seminorm[H^{\ell+3}(\Th)]{u},
      \label{eq:est.term22}
  \end{align}
  where we have used a Cauchy--Schwarz inequality and the norm equivalence \eqref{eq:norm.equivalence} to conclude.

  The estimate of $\term_{23}$ is obtained using Cauchy--Schwarz inequalities, the boudedness \eqref{eq:bound.PSigmaE} of the edge potential reconstructions, the norm equivalence \eqref{eq:norm.equivalence}, and the trace approximation properties of $\pi^{\ell+2}_{\mathcal{P},T}$ to get
  \begin{equation}\label{eq:est.term23}
    |\term_{23}|\lesssim h^{\ell+1}\seminorm[H^{\ell+3}(\Th)]{u}\norm[\btens{\Sigma},h]{\utens{\tau}_h}.
  \end{equation}
  Plugging \eqref{eq:est.term21}--\eqref{eq:est.term23} into \eqref{eq:term2.first}, we obtain $|\term_2|\lesssim h^{\ell+1}\seminorm[H^{\ell+3}(\Th)]{u}\norm[\btens{\Sigma},h]{\utens{\tau}_h}$ which, used in \eqref{eq:decompose.calE} together with \eqref{eq:term12}, concludes the proof of \eqref{eq:consistency.bound}.
\end{proof}


\section*{Acknowledgements}

The authors acknowledge the support of \emph{Agence Nationale de la Recherche} through the grant ANR-20-MRS2-0004 ``NEMESIS''.
Daniele Di Pietro also acknowledges the support of \emph{I-Site MUSE} through the grant ANR-16-IDEX-0006 ``RHAMNUS''.

\printbibliography

\end{document}